\documentclass[10pt,a4paper]{amsart}
\pdfoutput=1

\usepackage{amsmath,amssymb,amsthm}
\usepackage[utf8]{inputenc}
\usepackage{url}
\usepackage[shortlabels]{enumitem}
  \setitemize[1]{leftmargin=2em}
  \setenumerate[1]{leftmargin=*}
\usepackage[colorlinks=true,linktocpage=true,citecolor=blue,unicode,hyperindex,breaklinks]{hyperref}
\usepackage[capitalise]{cleveref}
\crefformat{equation}{(#2#1#3)}
\usepackage{color}
\usepackage{tikz}
\usetikzlibrary{matrix,arrows}
\usepackage{tikz-cd}
\usepackage[alphabetic]{amsrefs}
\usepackage{xstring}
\renewcommand{\eprint}[1]{\IfBeginWith{#1}{arXiv}{\href{https://arxiv.org/abs/#1}{#1}}{\href{#1}{#1}}}

\usepackage{eucal}
\makeatletter
\DeclareSymbolFontAlphabet{\amsbb}{AMSb}
\makeatother
\usepackage{mbboard}
\renewcommand{\mathbb}[1]{\amsbb{#1}}

\theoremstyle{plain}
\newtheorem{thm}{Theorem}[subsection]

\newtheorem{lemma}[thm]{Lemma}
\newtheorem{propn}[thm]{Proposition}
\newtheorem{cor}[thm]{Corollary}

\newtheorem*{thm*}{Theorem}
\newtheorem*{conjecture*}{Conjecture}
\newtheorem*{goal*}{Goal}
\newtheorem*{question*}{Question}
\newtheorem*{prethm*}{Pretheorem}
\theoremstyle{definition}
\newtheorem{defn}[thm]{Definition}
\newtheorem{ex}[thm]{Example}

\newtheorem{notation}[thm]{Notation}

\newtheorem{remark}[thm]{Remark}

\theoremstyle{remark}

\newcommand{\txt}[1]{\ensuremath{\text{\textup{#1}}}}
\newcommand{\catname}[1]{\txt{#1}}

\newcommand{\Cat}{\catname{Cat}}
\newcommand{\CatI}{\catname{Cat}_\infty}

\newcommand{\Opd}{\catname{Opd}}
\newcommand{\OpdI}{\Opd_{\infty}}

\newcommand{\Fun}{\txt{Fun}}
\newcommand{\Map}{\txt{Map}}

\newcommand{\op}{\txt{op}}
\DeclareMathOperator{\Seg}{Seg}

\newcommand{\IFF}{if and only if}
\newcommand{\icat}{$\infty$-category}
\newcommand{\icats}{$\infty$-categories}
\newcommand{\icatl}{$\infty$-categorical}
\newcommand{\igpd}{$\infty$-groupoid}
\newcommand{\igpds}{$\infty$-groupoids}

\newcommand{\iopd}{$\infty$-operad}
\newcommand{\iopds}{$\infty$-operads}

\newcommand{\ie}{i.e.\@}

\newcommand{\isoto}{\xrightarrow{\sim}}
\newcommand{\isofrom}{\xleftarrow{\sim}}

\newcommand{\xto}[1]{\xrightarrow{#1}}

\newcommand{\from}{\leftarrow}
\newcommand{\xfrom}[1]{\xleftarrow{#1}}

\newcommand{\blank}{\text{\textendash}}
\newcommand{\id}{\txt{id}}
\DeclareMathOperator{\colimP}{colim}
\newcommand{\colim}{\mathop{\colimP}}

\newcommand{\simp}{\bbDelta}
\newcommand{\Dop}{\simp^{\op}}
\newcommand{\angled}[1]{\langle #1 \rangle}

\makeatletter
\def\@tocline#1#2#3#4#5#6#7{\relax
  \ifnum #1>\c@tocdepth 
  \else
    \par \addpenalty\@secpenalty\addvspace{#2}%
    \begingroup \hyphenpenalty\@M
    \@ifempty{#4}{%
      \@tempdima\csname r@tocindent\number#1\endcsname\relax
    }{%
      \@tempdima#4\relax
    }%
    \parindent\z@ \leftskip#3\relax \advance\leftskip\@tempdima\relax
    \rightskip\@pnumwidth plus4em \parfillskip-\@pnumwidth
    #5\leavevmode\hskip-\@tempdima
      \ifcase #1
       \or \hskip -1em \or \hskip 1em \or \hskip 3em \else \hskip 5em \fi%
      #6\nobreak\relax
    \hfill\hbox to\@pnumwidth{\@tocpagenum{#7}}
      \par
    \nobreak
    \endgroup
  \fi}
\makeatother

\numberwithin{equation}{section}

\title{$\infty$-Operads as Symmetric Monoidal $\infty$-Categories}

\author{Rune Haugseng}
\address{Norwegian University of Science and Technology (NTNU),
  Trondheim, Norway}
\urladdr{http://folk.ntnu.no/runegha}

\author{Joachim Kock}
\address{Universitat Aut\`{o}noma de Barcelona and Centre de Recerca Matem\`atica, 
Spain}
\urladdr{http://mat.uab.cat/~kock/}

\date{\today. JK gratefully acknowledges support from grants MTM2016-80439-P
  (AEI/FEDER, UE) of Spain and 2017-SGR-1725 of Catalonia, and
  was also supported through the Severo Ochoa and Mar\'ia de
  Maeztu Program for Centers and Units of Excellence in R\&D
  grant number CEX2020-001084-M}

\newcommand{\xF}{\mathbb{F}}

\usepackage{extarrows}

\newcommand{\xint}{\txt{int}}
\newcommand{\act}{\txt{act}}
\newcommand{\el}{\txt{el}}

\newcommand{\SMC}{\txt{SMCat}}
\newcommand{\SMCI}{\SMC_{\infty}}
\newcommand{\CSeg}{\txt{CSeg}}
\newcommand{\Act}{\txt{Act}}
\newcommand{\Env}{\txt{Env}}
\newcommand{\CMon}{\txt{CMon}}

\DeclareMathOperator{\Sym}{Sym}
\newcommand{\PROPI}{\txt{PROP}_{\infty}}

\newcommand{\DF}{\simp_{\xF}}
\newcommand{\DFop}{\DF^{\op}}

\begin{document}

\begin{abstract}
  We use Lurie's symmetric monoidal envelope functor to give two new
  descriptions of $\infty$-operads: as certain symmetric monoidal
  $\infty$-categories whose underlying symmetric monoidal
  $\infty$-groupoids are free, and as certain symmetric monoidal
  $\infty$-categories equipped with a symmetric monoidal functor to
  finite sets (with disjoint union as tensor product). The latter
  leads to a third description of $\infty$-operads, as a localization of a
  presheaf $\infty$-category, and we use this to give a simple proof
  of the equivalence between Lurie's and Barwick's models for
  $\infty$-operads.
\end{abstract}

\maketitle

\tableofcontents

\section{Introduction}

The close relationship between symmetric monoidal categories and
(symmetric) operads goes back to the birth of operad theory in
algebraic topology: both operads \cite{MayGeomIter} and PROPs
\cite{MacLane:props}, which are special class of symmetric monoidal
categories, were introduced to describe homotopy-coherent algebraic
structures on topological spaces, and it was quickly realized that
operads could be viewed as a special kind of PROP (see for instance
the discussion in Adams's book \cite{AdamsInfLoop}*{\S 2.3}, or
\cite{KellyOpd}*{\S 7}). The relation has also been analysed in the
context of logic and computer science, notably by
Hermida~\cite{Hermida:repr-mult}.

In the setting of $\infty$-categories, the relationship between
$\infty$-operads and symmetric monoidal $\infty$-categories is more
pronounced, since in the approach of Lurie~\cite{HA} both notions are
defined as certain functors to the category $\xF_{*}$ of finite
pointed sets\footnote{Also known as Segal's category $\Gamma^{\op}$.};
for operads this corresponds to the construction of categories of
operators of May--Thomason~\cite{May-Thomason}. There is then an
evident forgetful functor
\[ U \colon \SMCI \to \OpdI \]
where $\OpdI$ is the \icat{} of \iopds{} and $\SMCI$ is that of
symmetric monoidal \icats{}. Lurie~\cite{HA}*{\S 2.2.4} established an
adjunction
\[
\begin{tikzcd}
\OpdI \ar[r, shift left, "\Env"] & \SMCI \ar[l, shift left, "U"]    ,
\end{tikzcd}
\]
where the left adjoint $\Env$ is given by an explicit construction,
the \emph{symmetric monoidal envelope} of an \iopd{}.  Neither of the
two functors is fully faithful, though, and so does not immediately
exhibit one notion as a special case of the other.\footnote{For $U$,
  this is because morphisms between symmetric monoidal \icats{} in
  $\OpdI$ corrrespond to \emph{lax} symmetric monoidal functors.}

In the present contribution, we exploit this adjunction to establish
new conceptually simple characterizations of $\infty$-operads, leading
to an easy proof of the equivalence between Lurie's and Barwick's
notions of $\infty$-operads.

\subsection{Overview}
\label{sec:results}

We start out by tweaking
the adjunction in two ways, so as to give
two new characterizations of $\infty$-operads in terms of symmetric monoidal 
$\infty$-categories:

\begin{thm}[Cf.~Propositions~\ref{propn:iopdimage} and 
  \ref{propn:counit-hereditary}]\label{thm:iopdenv} 
  The symmetric monoidal envelope gives an equivalence between
  $\infty$-operads in the sense of Lurie~\cite{HA}
  and
  \begin{enumerate}[(1)]
  \item symmetric monoidal \icats{} $\mathcal{C}^\otimes$ with a
    symmetric monoidal functor to $\xF^{\amalg}$, the category of
    finite sets with disjoint union as tensor product, such that
    \begin{enumerate}[(a)]
    \item every
      object of $\mathcal{C}$ is equivalent to a tensor product of
      objects that lie over the terminal object $\mathbf{1}$ in $\xF$,
    \item condition ($\ast$) below holds for
      any objects $x_{1},\ldots,x_{n}$ and $y_{1},\ldots,y_{m}$ that
      lie over $\mathbf{1}$.
    \end{enumerate}
  \item symmetric monoidal \icats{} $\mathcal{C}^\otimes$
    with a map of \igpds{} $X \to \mathcal{C}^{\simeq}$
    such that
    \begin{enumerate}[(a)]
    \item the underlying symmetric
      monoidal \igpd{} of $\mathcal{C}^{\otimes}$ is free on $X$, \ie{} the 
      induced morphism $\Sym(X)
      \simeq \coprod_{n =0}^{\infty} X^{\times n}_{h\Sigma_{n}} \to \mathcal{C}^\simeq$
      is an equivalence,
    \item condition ($\ast$) below holds
      for any objects $x_{1},\ldots,x_{n}$ and
      $y_{1},\ldots,y_{m}$ in $X$.
    \end{enumerate}
  \item[($\ast$)]
  The morphism
      \[ \coprod_{\phi \in
          \Map_{\xF}(\mathbf{n},\mathbf{m})} \prod_{i=1}^{m}
        \Map_{\mathcal{C}}\left(\bigotimes_{j \in \phi^{-1}(i)} x_{j},
          y_{i}\right) \to \Map_{\mathcal{C}}\left(\bigotimes_{j=1}^{n}
          x_{j}, \bigotimes_{i=1}^{m} y_{i}\right),\]
      given by tensoring maps together, is an equivalence.
    \end{enumerate}
  \end{thm}

\begin{remark}
  The condition ($\ast$) 
  appearing in both charaterizations can be traced back to the class
  of PROPs singled out by Boardman and Vogt in
  \cite{BoardmanVogt}*{Lemma 2.43}. More recently, it
   has been studied in different guises in the 1-categorical literature
   under the name of the \emph{hereditary condition}
   (cf.~\cite{Markl:0601129,Borisov-Manin:0609748,Kaufmann-Ward:1312.1269,Batanin-Kock-Weber:1510.08934},
   see also
  \cite{MelliesTabareau}).
  From that perspective, characterization (2) can be seen as the
  $\infty$-categorical version of the 
  equivalence of
  \cite{Batanin-Kock-Weber:1510.08934,Caviglia} between the \emph{Feynman 
    categories} of Kaufmann and Ward~\cite{Kaufmann-Ward:1312.1269}
  and coloured operads.
\end{remark}

\begin{remark}
  We are not aware of any direct precursor to characterization (1),
  but it fits well with Weber's $2$-categorical approach to operad
  theory~\cite{Weber:1412.7599}, where operads are essentially monads
  cartesian over the symmetric monoidal category monad.  Perhaps it
  should also be mentioned that in the theory of operadic categories
  of Batanin and Markl~\cite{Batanin-Markl:1404.3886}, which can be
  seen as a generalization of Barwick's idea of operator
  categories~\cite{BarwickOpCat}, it is an essential feature that
  everything lives over the category of finite sets.
\end{remark}  
  
Using the first characterization, we proceed to give a third:
Viewing \icats{} as complete Segal spaces, we can describe symmetric
monoidal \icats{} over $\xF^{\amalg}$ as functors $\mathcal{F} \to
\mathcal{S}$ satisfying completeness and Segal conditions for a
certain category $\mathcal{F}$, giving an equivalence
\[ \SMC_{\infty/\xF^{\amalg}} \simeq
  \CSeg_{\mathcal{F}}(\mathcal{S}) \subseteq \Fun(\mathcal{F},
  \mathcal{S}). \]
We can then identify \iopds{} as those complete Segal
$\mathcal{F}$-spaces that satisfy some further conditions:
\begin{thm}[\cref{thm:opdascsegF}]
  There is an equivalence
  \[ \OpdI \isoto \CSeg'_{\mathcal{F}}(\mathcal{S}),\]
  where $\CSeg'_{\mathcal{F}}(\mathcal{S})$ is a certain full
  subcategory of $\CSeg_{\mathcal{F}}(\mathcal{S})$.
\end{thm}

Since $\CSeg'_{\mathcal{F}}(\mathcal{S})$ is by definition an
accessible localization of a presheaf \icat{}, this result implies in
particular that $\OpdI$ is a presentable \icat{}, without appealing to
a presentation of $\OpdI$ by a model category.
Our main motivation for this characterization, however, is that it is
a key ingredient in the final result of the paper: via an obvious
comparison functor between $\mathcal{F}$ and Barwick's category
$\DFop$, we obtain with very little work an equivalence between
$\CSeg'_{\mathcal{F}}(\mathcal{S})$ and Barwick's definition of
\iopds{} as presheaves on $\DF$ satisfying Segal and completeness
conditions. Thus we get a simple proof of the equivalence between
Lurie's and Barwick's approaches to $\infty$-operads:
\begin{cor}
  There is an equivalence of \icats{}
  \[ \OpdI \simeq \CSeg_{\DFop}(\mathcal{S})\]
  between Lurie's and Barwick's models for \iopds{}.
\end{cor}
This theorem was already proved by Barwick~\cite{BarwickOpCat} by a
rather different method (which involves studying the nerve adjunction
for a functor $\DF \to \OpdI$). Note that Barwick's result is
substantially more general than ours, giving an equivalence between
two definitions of \iopds{} over any \emph{perfect operator
  category}, where the explicit description of the monoidal envelope
required for our proof typically fails.\footnote{Our approach does
  also work in the particular case of non-symmetric (or planar)
  \iopds{}, however.}

\subsection{Some Basic Notation}
This paper is written in the language of \icats{}, and all terms such
as (co)limits and commutative diagrams should be understood in their
fully homotopy-coherent/\icatl{} sense.
\begin{itemize}
\item $\xF$ is (a skeleton of) the category of finite sets, with
  objects $\mathbf{n} = \{1,\ldots,n\}$ ($n = 0,1,\ldots$).
\item $\mathcal{S}$ is the \icat{} of (small) \igpds{}/spaces/homotopy types.
\item $\CatI$ is the \icat{} of (small) \icats{}.
\item If $\mathcal{C}$ is an \icat{},
  $\Cat_{\infty/\mathcal{C}}^{\txt{L}}$ denotes the full subcategory of
  the overcategory $\Cat_{\infty/\mathcal{C}}$ spanned by the left
  fibrations to $\mathcal{C}$.
\item If $\mathcal{C}$ is an \icat{},
  $\Cat_{\infty/\mathcal{C}}^{\txt{coc}}$ denotes the subcategory of
  the overcategory $\Cat_{\infty/\mathcal{C}}$ with objects the
  cocartesian fibrations to $\mathcal{C}$ and morphisms the functors
  over $\mathcal{C}$ that preserve cocartesian morphisms.
\item If $\mathcal{C}$ is an \icat{}, we write $\mathcal{C}^{\simeq}$
  for its underlying \igpd{}, \ie{} the subcategory containing only equivalences.
\end{itemize}

\section{From Lurie's $\infty$-Operads to Symmetric Monoidal
  $\infty$-Categories}

In this section we first review the basic notions of commutative
monoids in \icats{} (and in particular symmetric monoidal \icats{}) in
\S\ref{subsec:cmon} and \iopds{} (in the sense of \cite{HA}) in
\S\ref{subsec:Luriopd}. Then we recall the symmetric monoidal envelope
of an \iopd{} in \S\ref{subsec:env} before we study its image and
prove \cref{thm:iopdenv} in \S\ref{subsec:envimg}.

\subsection{Commutative Monoids and Symmetric Monoidal $\infty$-Categories}\label{subsec:cmon}
We now recall the \icatl{} notion of commutative
monoid, originally introduced by Segal~\cite{SegalCatCohlgy}. As a
special case, this also gives the definition of symmetric monoidal \icats{}.

\begin{notation}
  We write $\xF_{*}$ for (a skeleton of) the category of finite
  pointed sets. We will make use of two equivalent descriptions of this category:
  \begin{enumerate}[(1)]
  \item The objects of $\xF_{*}$ are the pointed sets $\angled{n} =
    (\{0,1,\ldots,n\}, 0)$ ($n = 0,1,\ldots$) and the morphisms $\angled{n} \to
    \angled{m}$ are the functions that preserve the base point.
  \item The objects of $\xF_{*}$ are the sets $\mathbf{n} =
    \{1,\ldots,n\}$ ($n = 0,1,\ldots$), and morphisms from
    $\mathbf{n}$ to $\mathbf{m}$ are isomorphism 
	classes\footnote{In fact the groupoid of such spans is discrete, so from an \icatl{} viewpoint taking isomorphism classes doesn't do anything.} of \emph{spans}
    \[ \mathbf{n} \hookleftarrow \mathbf{x} \to \mathbf{m} \]
    where the backwards map is injective.  Spans are composed by taking
    pullbacks.
  \end{enumerate}
  To pass between these two descriptions, note that giving a pointed map
  $\angled{n} \to \angled{m}$ is the same thing as giving a map of
  sets $I \to \mathbf{m}$ where $I$ is the subset of $\angled{n}$ that
  is not mapped to the base point. (Up to unique isomorphism, 
  $I$ can be replaced by an object in the chosen skeleton.)
\end{notation}

\begin{defn}\label{defn:F*fact}
  A morphism $\phi \colon \angled{n} \to \angled{m}$ in $\xF_{*}$ is
  \emph{active} if $\phi^{-1}(0) = \{0\}$ and \emph{inert} if
  $\phi|_{\angled{n}\setminus \phi^{-1}(0)}$ is an isomorphism. The inert
  and active morphisms form a factorization system on $\xF_{*}$; in
  particular, every morphism factors uniquely up to isomorphism as an inert morphism
  followed by an active morphism.
\end{defn}
\begin{remark}
  In terms of the second description of $\xF_{*}$, a span
  \[\mathbf{n} \hookleftarrow \mathbf{k} \to \mathbf{m}\]
  is active if the inclusion $\mathbf{n} \hookleftarrow \mathbf{k}$
  is an isomorphism, and inert if the map $\mathbf{k} \to \mathbf{m}$
  is an isomorphism.
\end{remark}

\begin{notation}
  For $\angled{n} \in \xF_{*}$ and $i = 1,\ldots,n$, we write
  $\rho_{i} \colon \angled{n} \to \angled{1}$ for the inert map given
  by
  \[ \rho_{i}(j) =
    \begin{cases}
      0, & i \neq j, \\
      1, & i = j.
    \end{cases}
  \]
  Alternatively, this is the span
  \[ \mathbf{n} \hookleftarrow \{i\} \xto{=}  \mathbf{1}.\]
\end{notation}

\begin{defn}\label{defn:cmon}
  Let $\mathcal{C}$ be an \icat{} with finite products. A
  \emph{commutative monoid} in $\mathcal{C}$ is a functor $M \colon
  \xF_{*} \to \mathcal{C}$ such that for every $\angled{n} \in
  \xF_{*}$, the natural morphism
  \[ M(\angled{n}) \to \prod_{i=1}^{n} M(\angled{1}), \]
  determined by the maps $\rho_{i}$, is an equivalence. We write
  $\CMon(\mathcal{C})$ for the full subcategory of $\Fun(\xF_{*},
  \mathcal{C})$ spanned by the commutative monoids.
\end{defn}

\begin{defn}
  A \emph{symmetric monoidal \icat{}} is a commutative monoid in the
  \icat{} $\CatI$ of \icats{}.
\end{defn}

\begin{remark}
  Equivalently, using the straightening equivalence between functors
  to $\CatI$ and cocartesian fibrations, we can view a symmetric
  monoidal \icat{} as a cocartesian fibration over $\xF_{*}$.
\end{remark}

\begin{notation}
  We write $\SMCI$ for the \icat{} of symmetric monoidal \icats{}. This
  can be viewed as a full subcategory of either $\Fun(\xF_{*}, \CatI)$
  or $\Cat_{\infty/\xF_{*}}^{\txt{coc}}$.
\end{notation}

\subsection{Lurie's $\infty$-Operads}\label{subsec:Luriopd}
Here we recall Lurie's definition of \iopds{} from \cite{HA}*{\S
  2.1.1} and its relation to symmetric monoidal \icats{}.
\begin{defn}
  An \emph{\iopd{}} is a functor $\pi \colon \mathcal{O} \to \xF_{*}$ such
  that:
  \begin{enumerate}[(1)]
  \item $\mathcal{O}$ has $\pi$-cocartesian morphisms over inert
    morphisms in $\xF_{*}$.
  \item For every $\angled{n} \in \xF_{*}$, the functor
    \[ \mathcal{O}_{\angled{n}} \to \prod_{i = 1}
      \mathcal{O}_{\angled{1}}, \]
    given by cocartesian transport along the maps $\rho_{i} \colon
    \angled{n} \to \angled{1}$, is an equivalence.
  \item For $X \in \mathcal{O}_{\angled{n}}$, if $\overline{\rho}_{i}
    \colon X \to X_{i}$ is a cocartesian morphism over $\rho_{i}$ ($i
    = 1,\ldots,n$), for any $Y \in \mathcal{O}_{\angled{m}}$ the commutative square
    \[
      \begin{tikzcd}
        \Map_{\mathcal{O}}(Y, X) \arrow{r}{(\overline{\rho}_{i,*})}
        \arrow{d} & \prod_{i=1}^{n} \Map_{\mathcal{O}}(Y,X_{i})
        \arrow{d} \\
        \Map_{\xF_{*}}(\angled{m},\angled{n}) \arrow{r}{(\rho_{i,*})}
        & \prod_{i=1}^{n} \Map(\angled{m}, \angled{1})
      \end{tikzcd}
    \]
    is a pullback square.
  \end{enumerate}
\end{defn}

\begin{remark}\label{rmk:smiopd}
  It is not hard to see that a symmetric monoidal \icat{}, viewed as a
  cocartesian fibration over $\xF_{*}$, is precisely an \iopd{} that
  is also a cocartesian fibration.
\end{remark}

 \begin{defn}
  If $\pi \colon \mathcal{O} \to \xF_{*}$ is an \iopd{}, we say a
  morphism in $\mathcal{O}$ is \emph{inert} if it is a cocartesian
  morphism over an inert morphism in $\xF_{*}$, and \emph{active} if
  it lies over an active morphism in $\xF_{*}$. By
  \cite{HA}*{Proposition 2.1.2.5}, the inert and active morphisms form
  a factorization system on $\mathcal{O}$.
\end{defn}

\begin{defn}
  If $p \colon \mathcal{O} \to \xF_{*}$ and $q \colon \mathcal{P} \to
  \xF_{*}$ are \iopds{}, then a \emph{morphism of \iopds{}} from
  $\mathcal{O}$ to $\mathcal{P}$ is a commutative triangle
  \[
    \begin{tikzcd}
      \mathcal{O} \arrow{rr}{f} \arrow{dr}[swap]{p} & & \mathcal{P}
      \arrow{dl}{q} \\
       & \xF_{*}
    \end{tikzcd}
  \]
  such that $f$ preserves inert morphisms. We write $\OpdI$ for the
  subcategory of $\Cat_{\infty/\xF_{*}}$ whose objects are the
  \iopds{} and whose morphisms are the morphisms of \iopds{}.
\end{defn}

\begin{remark}
  By \cref{rmk:smiopd}, if we view symmetric monoidal \icats{} as
  cocartesian fibrations then the subcategory of
  $\Cat_{\infty/\xF_{*}}$ corresponding to $\SMCI$ is contained in
  $\OpdI$, so that we have a forgetful functor
  $U \colon \SMCI \to \OpdI$. Note that this is not fully faithful: a
  morphism in $\SMCI$ is required to preserve all cocartesian
  morphisms and corresponds to a symmetric monoidal functor, while a
  morphism in $\OpdI$ is only required to preserve the cocartesian
  morphisms that lie over inert maps in $\xF_{*}$. Such a morphism can be
  interpreted as a lax symmetric monoidal functor.
\end{remark}

\begin{defn}
  Suppose $\mathcal{O}$ is an \iopd{}. Given a full subcategory
  $\mathcal{C}$ of $\mathcal{O}_{\angled{1}}$, the full subcategory of
  $\mathcal{O}$ spanned by the objects that lie in
  $\mathcal{C}^{\times n} \subseteq \mathcal{O}_{\angled{1}}^{\times
    n}$ under the equivalence $\mathcal{O}_{\angled{1}}^{\times
    n} \simeq \mathcal{O}_{\angled{n}}$, for all $n$, is again an
  \iopd{}. We refer to this as the \emph{full suboperad} of
  $\mathcal{O}$ spanned by the objects in $\mathcal{C}$.
\end{defn}

\subsection{Symmetric Monoidal Envelopes}\label{subsec:env}
In this subsection we recall the construction of symmetric monoidal
envelopes from \cite{HA}*{\S 2.2.4}.

\begin{notation}
  Let $\Act(\xF_{*})$ denote the full subcategory of the arrow
  category $\xF_{*}^{[1]}$ whose objects are the active morphisms. We
  write $s,t \colon \Act(\xF_{*}) \to \xF_{*}$ for the source and
  target projections. If $i \colon \xF_{*} \to \Act(\xF_{*})$ denotes
  the functor that assigns to each object its identity map, then
  $si = ti = \id_{\xF_{*}}$.
\end{notation}

\begin{defn}
  For $\mathcal{O}$ an \iopd{}, we write $\Env(\mathcal{O}) \to
  \xF_{*}$ for the fibre product $\mathcal{O} \times_{\xF_{*}}
  \Act(\xF_{*})$ along $s$, with the map to $\xF_{*}$ induced by
  $t$. This gives a functor $\Env \colon \OpdI \to \Cat_{\infty/\xF_{*}}$.
\end{defn}

\begin{lemma}\label{lem:iOcart}
  There is a natural pullback square
    \begin{equation}
      \label{eq:iOpb}
    \begin{tikzcd}
      \mathcal{O} \arrow{d} \arrow{r}{i_{\mathcal{O}}} & \Env(\mathcal{O}) \arrow{d} \\
      \xF_{*} \arrow{r}{i} & \Act(\xF_{*}).
    \end{tikzcd}
    \end{equation} 
\end{lemma}
\begin{proof}
  By definition we have a commutative diagram
  \[
    \begin{tikzcd}
      {} & \Env(\mathcal{O}) \arrow{r} \arrow{d} & \mathcal{O}
      \arrow{d} \\
      \xF_{*} \arrow[bend right]{rr}{\id} \arrow{r}{i} & \Act(\xF_{*})
      \arrow{r}{s} & \xF_{*}
    \end{tikzcd}
  \]
  where the square is a pullback square. The pullback along $i$ is therefore
  indeed given by $\mathcal{O} \to \xF_{*}$.
\end{proof}

\begin{remark}
  Since $ti = \id$, we can view $i_{\mathcal{O}}$ as a natural map
  $\mathcal{O} \to \Env(\mathcal{O})$ over $\xF_{*}$.
\end{remark}

\begin{thm}[\cite{HA}*{Propositions 2.2.4.4 and 2.2.4.9}]
  The construction $\Env$ gives a functor $\OpdI \to \SMCI$, which is
  left adjoint to the forgetful functor $U \colon \SMCI \to \OpdI$,
  with unit transformation given by the natural maps $i_{(\blank)}$. \qed
\end{thm}

\begin{remark}\label{rmk:Envdesc}
  If $\pi \colon \mathcal{O} \to \xF_{*}$ is an \iopd{}, an object of
  $\Env(\mathcal{O})$ over $\angled{n}$ is given by an object
  $X \in \mathcal{O}$ together with an active morphism $\alpha \colon \pi(X) \to
  \angled{n}$ in $\xF_{*}$. A morphism $(X,\alpha) \to (Y,\beta)$
  in $\Env(\mathcal{O})$ is given by a morphism $\phi \colon X \to Y$ in $\mathcal{O}$ and a
  commutative square
  \[
    \begin{tikzcd}
      \pi(X) \arrow{r}{\pi(\phi)} \arrow{d}{\alpha} & \pi(Y)
      \arrow{d}{\beta} \\
      \angled{n} \arrow{r}{\psi} & \angled{m}.
    \end{tikzcd}
  \]
  If the underlying map $\psi$ is active, then the uniqueness of
  factorizations forces $\phi$ to be an active map in
  $\mathcal{O}$. In particular, since every object of
  $\xF_{*}$ has a unique active map to $\angled{1}$, the underlying \icat{}
  $\Env(\mathcal{O})_{\angled{1}}$ can be identified with the
  subcategory $\mathcal{O}^{\act}$ containing only the active maps in
  $\mathcal{O}$. Given a morphism $\psi \colon \angled{n} \to
  \angled{m}$, the cocartesian morphism $(X,\alpha) \to \psi_{!}(X,\alpha)$ can be
  described as follows:
  The inert-active factorization of $\psi \circ \alpha$ gives a
  commutative square
  \begin{equation}
    \label{eq:iafactsq}
        \begin{tikzcd}
      \pi(X) \arrow{d}{\alpha} \arrow{r}{i} & \angled{k}
      \arrow{d}{a} \\
      \angled{n} \arrow{r}{\psi} & \angled{m}
    \end{tikzcd}
  \end{equation}
  where $i$ is inert and $a$ is active. Since $\mathcal{O}$ is an
  \iopd{} there is a cocartesian morphism $X \to i_{!}X$ in
  $\mathcal{O}$, and $\psi_{!}(X,\alpha)$ is given by $(i_{!}X, a)$
  with the cocartesian morphism in $\mathcal{O}$ together with the
  commutative square \cref{eq:iafactsq}. In particular, if we think of
  objects of
  $\Env(\mathcal{O})_{\angled{1}} \simeq \mathcal{O}^{\act}$ as
  sequences of objects in $\mathcal{O}_{\angled{1}}$, then their
  tensor product is given by concatenation. Given a morphism of
  \iopds{} $F \colon \mathcal{O} \to \mathcal{C}^{\otimes}$, where
  $\mathcal{C}^{\otimes}$ is a symmetric monoidal \icat{}, the
  canonical extension of $F$ to a symmetric monoidal functor
  $\Env(\mathcal{O}) \to \mathcal{C}^{\otimes}$ takes $(X,\alpha)$ to
  the codomain $\alpha_{!}F(X)$ of the cocartesian morphism from
  $F(X)$ over $\alpha$.
\end{remark}

\begin{remark}\label{rmk:EnvF*}
  For the terminal \iopd{} $\xF_{*}$ we can describe $\Env(\xF_{*})$
  even more explicitly: The underlying category
  $\Env(\xF_{*})_{\angled{1}} \simeq \xF_{*}^{\act}$ we can identify
  with the category $\xF$ of finite sets, and under this
  identification the ``concatenation'' symmetric monoidal structure
  corresponds to disjoint union. In other words, the symmetric
  monoidal \icat{} $\Env(\xF_{*})$ is
  equivalent to the coproduct symmetric monoidal structure on $\xF$,
  that is to say
  \[\Env(\xF_{*}) \simeq \xF^{\amalg}.\]
\end{remark}

\subsection{Two Descriptions of $\infty$-Operads via
  Envelopes}\label{subsec:envimg}

In this section we will use the symmetric monoidal envelope functor to
give two descriptions of \iopds{} in terms of symmetric monoidal
\icats{} and thus prove \cref{thm:iopdenv}.

For the first description we want to consider
symmetric monoidal \icats{} equipped with a map to $\xF^{\amalg}$. Since
we saw in \cref{rmk:EnvF*} that $\Env$ takes the terminal \iopd{}
$\xF_{*}$ to $\xF^{\amalg}$, we have a functor
\[ \Env' \colon \OpdI \to \SMC_{\infty/\xF^{\amalg}} \] that just
applies $\Env$ to the unique map to the terminal object in $\OpdI$.

\begin{lemma}
  The functor $\Env'$ has a right adjoint
  \[ U' \colon \SMC_{\infty/\xF^{\amalg}} \to \OpdI,\] given by
  applying the forgetful functor $U$ and then pulling back along the
  unit map $\xF_{*} \to \xF^{\amalg}$.
\end{lemma}
\begin{proof}
  This is a special case of \cite{HTT}*{Proposition 5.2.5.1}.
\end{proof}

\begin{remark}
In other words, if $\mathcal{C}^{\otimes}$ is a symmetric monoidal
\icat{} over $\xF_{*}$ then we have a pullback square
\[
  \begin{tikzcd}
    U'(\mathcal{C}^{\otimes}) \arrow{r} \arrow{d} &
    \mathcal{C}^{\otimes} \arrow{d} \\
    \xF_{*} \arrow{r}{i} & \xF^{\amalg}.
  \end{tikzcd}
\]  
\end{remark}

\begin{propn}\label{EnvoverFff}
  $\Env' \colon \OpdI \to \SMC_{\infty/\xF^{\amalg}}$ is fully faithful.
\end{propn}
\begin{proof}
  It suffices to show that the unit transformation $\id \to U'\Env'$
  is an equivalence, which follows from the pullback square
  \cref{eq:iOpb} in \cref{lem:iOcart}.
\end{proof}

\begin{notation}
  For a symmetric monoidal functor $\mathcal{C}^{\otimes} \to
  \xF^{\amalg}$, we write \[\mathcal{C}^{\otimes}_{(1)} :=
  U'\mathcal{C}^{\otimes}\] for the pullback along $\xF_{*} \to
  \xF^{\amalg}$, and $\mathcal{C}_{(1)} :=
  (\mathcal{C}^{\otimes}_{(1)})_{\angled{1}}$ for the fibre of
  $\mathcal{C}$ over $\mathbf{1} \in \xF$. Note that since $\mathbf{1}$
  has no endomorphisms in $\xF$, the inclusion $\mathcal{C}_{(1)} \to
  \mathcal{C}$ exhibits $\mathcal{C}_{(1)}$ as a full subcategory, and
  thus $\mathcal{C}^{\otimes}_{(1)}$ is the full suboperad of
  $\mathcal{C}^{\otimes}$ spanned by objects of $\mathcal{C}$ that lie
  over $\mathbf{1}$.
\end{notation}

\begin{cor}
  $\OpdI$ is equivalent to the full subcategory of
  $\SMC_{\infty/\xF^{\amalg}}$ consisting of symmetric monoidal
  \icats{} $\mathcal{C}^{\otimes}$ over $\xF^{\amalg}$ such that the
  counit map
  \begin{equation}
    \label{eq:envcounit} 
        \Env(\mathcal{C}^{\otimes}_{(1)}) \to \mathcal{C}^{\otimes}
  \end{equation}    
  is an equivalence.\qed
\end{cor}

We will now describe this full subcategory more explicitly:
\begin{propn}\label{propn:iopdimage} 
  For $\mathcal{C}^{\otimes} \in \SMC_{\infty/\xF^{\amalg}}$, the
  counit map \cref{eq:envcounit}  is an equivalence \IFF{} the
  two following conditions hold:
  \begin{enumerate}[(1)]
  \item\label{escond} Every object in $\mathcal{C}$ is equivalent to a tensor
    product $x_{1} \otimes \cdots \otimes x_{n}$ with $x_{i}
    \in \mathcal{C}_{(1)}$.
    \item\label{ffcond} Given objects $x_{1},\ldots,x_{n}$ and
      $y_{1},\ldots,y_{m}$ in $\mathcal{C}_{(1)}$, the morphism
      \[ \coprod_{\phi \in
          \Map_{\xF}(\mathbf{n},\mathbf{m})} \prod_{i=1}^{m}
        \Map_{\mathcal{C}}\left(\bigotimes_{j \in \phi^{-1}(i)} x_{j},
          y_{i}\right) \to \Map_{\mathcal{C}}\left(\bigotimes_{j=1}^{n}
          x_{j}, \bigotimes_{i=1}^{m} y_{i}\right),\]
      given by tensoring maps together, is an equivalence.
  \end{enumerate}
\end{propn}

\begin{remark}
  Condition (2) is the so-called ``hereditary condition'' considered in
  \cite{Markl:0601129,Borisov-Manin:0609748,MelliesTabareau,Kaufmann-Ward:1312.1269,Batanin-Kock-Weber:1510.08934}.
\end{remark}

\begin{proof}
  A morphism in $\SMC_{\infty/\xF^{\amalg}}$ is an equivalence \IFF{}
  the functor of underlying \icats{} is an equivalence. It therefore
  suffices to show that the given conditions are equivalent to the
  functor
  \[ \epsilon \colon (\mathcal{C}^{\otimes}_{(1)})^{\act} \simeq
    \Env(\mathcal{C}^{\otimes}_{(1)})_{\angled{1}} \to \mathcal{C}\]
  being an equivalence.

  An object of $(\mathcal{C}^{\otimes}_{(1)})^{\act}$ can be described
  as a list
  $(x_{1},\ldots,x_{n})$ where each $x_{i}$ is an object of
  $\mathcal{C}_{(1)}$, and $\epsilon(x_{1},\ldots,x_{n})$ is the
  tensor product $x_{1} \otimes \cdots \otimes x_{n}$. Condition
  \ref{escond} therefore corresponds precisely to $\epsilon$ being
  essentially surjective.

  A morphism in $(\mathcal{C}^{\otimes}_{(1)})^{\act}$ from
  $(x_{1},\ldots,x_{n})$ to $(y_{1},\ldots,y_{m})$ is given by a map
  $\phi \colon \mathbf{n} \to \mathbf{m}$ in $\xF$ together with a morphism
  $f_{i} \colon \bigotimes_{j \in \phi^{-1}(i)} x_{j} \to y_{i}$ lying over the unique
  map $\phi^{-1}(i) \to \mathbf{1}$, for every $i = 1,\ldots,m$. The
  functor $\epsilon$ takes this to the morphism
  \[ \bigotimes_{i=1}^{m} f_{i} \colon \bigotimes_{j=1}^{n} x_{j}
    \isoto \bigotimes_{i=1}^{m} \left(\bigotimes_{j \in \phi^{-1}(i)}
      x_{j} \right) \to \bigotimes_{i=1}^{m} y_{i}\]
  in $\mathcal{C}$. (The unnamed equivalence is explicit: it is
  given by permutation of tensor factors according to the bijection $\sigma_\phi \colon 
  \mathbf{n}\isoto\mathbf{n}$ obtained by factoring $\phi = \lambda_\phi \circ \sigma_\phi$
  where $\lambda_\phi$ is monotone and $\sigma_\phi$ is a bijection monotone on 
  fibres. For the present purposes, these permutations do not play any significant role.)
  In other words, there is an equivalence
  \[ \Map_{(\mathcal{C}^{\otimes}_{(1)})^{\act}}((x_{1},\ldots,x_{n}),
    (y_{1},\ldots,y_{m})) \simeq \coprod_{\phi \in
      \Map_{\xF}(\mathbf{n},\mathbf{m})} \prod_{i=1}^{m}
    \Map_{\mathcal{C}}\left(\bigotimes_{j \in \phi^{-1}(i)} x_{j},
    y_{i}\right), \]
  and the map to $\Map_{\mathcal{C}}(\bigotimes_{j=1}^{n} x_{j},
  \bigotimes_{i=1}^{m} y_{i})$ is given by tensoring maps
  together (after appropriately permuting tensor factors). Condition \ref{ffcond} therefore corresponds precisely to
  $\epsilon$ being fully faithful.
\end{proof}

\begin{remark}\label{rmk:iopdcondprime}
  Using the functor $\mathcal{C} \to \xF$, the morphism in
  \ref{ffcond} fits in a commutative triangle
  \[
    \begin{tikzcd}[cramped,column sep=0pt]
      \displaystyle \coprod_{\phi \in
          \Map_{\xF}(\mathbf{n},\mathbf{m})} \prod_{i=1}^{m}
        \Map_{\mathcal{C}}\left(\bigotimes_{j \in \phi^{-1}(i)} x_{j},
          y_{i}\right) \arrow{dr} \arrow{rr} & & \displaystyle \Map_{\mathcal{C}}\left(\bigotimes_{j=1}^{n}
          x_{j}, \bigotimes_{i=1}^{m} y_{i}\right) \arrow{dl} \\
         & \Map_{\xF}(\mathbf{n},\mathbf{m}),
    \end{tikzcd}
  \]
  so that passing to fibres we can equivalently phrase \ref{ffcond} as:
  \begin{itemize}
  \item[($2'$)]   For every morphism $\phi \colon \mathbf{n} \to \mathbf{m}$ in $\xF$,
    the map
    \[ \prod_{i=1}^{m} \Map_{\mathcal{C}}\left(\bigotimes_{j \in \phi^{-1}(i)} x_{j},
        y_{i}\right) \to \Map_{\mathcal{C}}\left(\bigotimes_{j=1}^{n}
        x_{j}, \bigotimes_{i=1}^{m} y_{i}\right)\!\!{}_{\phi}, \]
    given by tensoring morphisms, is an equivalence. 
  \end{itemize}
  In particular, taking $\phi$ to be $\id_{\mathbf{n}}$ we have
  equivalences
  \[ \prod_{i=1}^{n} \Map_{\mathcal{C}_{(1)}}(x_{i},y_{i}) \isoto
    \Map_{\mathcal{C}}\left(\bigotimes_{i=1}^{n} x_{i},
      \bigotimes_{i=1}^{n} y_{i}\right)\!\!{}_{\id_{\mathbf{n}}}
    \simeq \Map_{\mathcal{C}_{(n)}}\left(\bigotimes_{i=1}^{n} x_{i},
      \bigotimes_{i=1}^{n} y_{i}\right),\] where $\mathcal{C}_{(n)}$
  is the fibre of $\mathcal{C} \to \xF$ at $\mathbf{n}$. This says
  precisely that the functor
  $\mathcal{C}_{(1)}^{\times n} \to \mathcal{C}_{(n)}$ is fully
  faithful. On the other hand, condition \ref{escond} amounts to
  requiring the same functors to be essentially surjective. In the
  presence of condition \ref{ffcond} (or equivalently ($2'$)) we can
  therefore replace \ref{escond} by
  \begin{itemize}
  \item[($1'$)] For every $n$, the functor $\mathcal{C}_{(1)}^{\times
      n} \to \mathcal{C}_{(n)}$, induced by the tensor product, is an
    equivalence.
  \end{itemize}
  Alternatively, since the full faithfulness of this functor is also
  part of (2), we can replace (1) with
    \begin{itemize}
    \item[($1''$)] For every $n$, the map of spaces
      $(\mathcal{C}^{\simeq}_{(1)})^{\times n} \to
      \mathcal{C}^{\simeq}_{(n)}$, induced by the tensor product, is
      an equivalence.
    \end{itemize}
    Finally, note that we can reformulate ($2'$) for all objects at
    once as:
    \begin{itemize}
    \item[($2''$)] For every morphism $\phi \colon \mathbf{n} \to \mathbf{m}$ in $\xF$,
      the map \[ \prod_{i=1}^{m} \Map(\Delta^{1},
        \mathcal{C})_{\mathbf{n}_{i} \to \mathbf{1}} \to
        \Map(\Delta^{1}, \mathcal{C})_{\phi}, \]
    given by tensoring morphisms, is an equivalence.
  \end{itemize}
  This is equivalent to ($2'$) since we have a commutative square
  \[
    \begin{tikzcd}
      \prod_{i=1}^{m} \Map(\Delta^{1},
        \mathcal{C})_{\mathbf{n}_{i} \to \mathbf{1}} \arrow{r}
        \arrow{d} & \Map(\Delta^{1}, \mathcal{C})_{\phi} \arrow{d} \\
        \prod_{i=1}^{m} \mathcal{C}_{(n_{i})}^{\simeq} \times
        \mathcal{C}_{(1)}^{\simeq} \arrow{r}{\sim} &
        \mathcal{C}_{(n)}^{\simeq} \times \mathcal{C}_{(m)}^{\simeq},
    \end{tikzcd}
  \]
  where the bottom horizontal map is an equivalence by ($1''$), and
  the maps on fibres are those in ($2'$).
\end{remark}

Now we turn to the second description:
\begin{defn}
  The forgetful functor $\CMon(\mathcal{S}) \to \mathcal{S}$ has a
  left adjoint $\Sym \colon \mathcal{S} \to \CMon(\mathcal{S})$. Since
  the underlying \igpd{} functor $(\blank)^{\simeq} \colon \CatI \to
  \mathcal{S}$ preserves products, it induces a functor $\CMon(\CatI)
  \to \CMon(\mathcal{S})$, and we define $\PROPI$ as the pullback
  \[
    \begin{tikzcd}
      \PROPI \arrow{r} \arrow{d} & \CMon(\CatI) \arrow{d}{(\blank)^{\simeq}} \\
      \mathcal{S} \arrow{r}{\Sym} & \CMon(\mathcal{S}).
    \end{tikzcd}
  \]
  An object of $\PROPI$ is thus a symmetric monoidal \icat{}
  $\mathcal{C}$ together with an \igpd{} $X$ and an equivalence of
  symmetric monoidal \igpds{} $\Sym X \simeq \mathcal{C}^{\simeq}$.
\end{defn}

\begin{remark}
  As the name suggests, we think of the objects of $\PROPI$ as a good
  \icatl{} analogue of the classical notion of PROPs, but we will not
  try to justify this here. Note, however, that PROPs are usually
  defined to be symmetric monoidal categories whose underlying
  \emph{set} of objects is a free commutative monoid, while our
  definition corresponds for ordinary categories to having a free
  underlying symmetric monoidal \emph{groupoid}. This condition does
  have the advantage of being invariant under equivalence, whereas
  with the more traditional definition \emph{every} symmetric monoidal
  category is equivalent to a PROP (since every commutative monoid in
  sets admits a surjective map from a free one). On the other hand,
  this probably means that our \icat{} $\PROPI$ does not correspond to
  the Quillen model structure on simplicial PROPs of Hackney and
  Robertson~\cite{HackneyRobertsonProp}.
\end{remark}

\begin{propn}\label{propn:Enveqisfree}
  For any \iopd{} $\mathcal{O}$, the functor $i_{\mathcal{O}} \colon
  \mathcal{O} \to \Env(\mathcal{O})$ restricts to a morphism of
  \igpds{} $\mathcal{O}_{\angled{1}}^{\simeq} \to
  \Env(\mathcal{O})_{\angled{1}}^{\simeq}$ that is adjoint to an
  equivalence of commutative monoids
  \[ \Sym(\mathcal{O}_{\angled{1}}^{\simeq}) \isoto
  \Env(\mathcal{O})_{\angled{1}}^{\simeq}.\]
\end{propn}
\begin{proof}
  Consider the subcategory $\mathcal{O}^{\xint}$ of $\mathcal{O}$
  containing only the inert morphisms. This is also an \iopd{}, and for any
  \iopd{} $\mathcal{P}$ we have equivalences
  \[ \Map_{\OpdI}(\mathcal{O}^{\xint}, \mathcal{P}) \simeq
    \Map_{\Cat_{\infty/\xF_{*}^{\xint}}^{\txt{coc}}}(\mathcal{O}^{\xint},
    \mathcal{P} \times_{\xF_{*}} \xF_{*}^{\xint}) \simeq
    \Map_{\CatI}(\mathcal{O}^{\simeq}_{\angled{1}},
    \mathcal{P}_{\angled{1}}),\]
  where the first equivalence is obtained by pulling back along
  $\xF_{*}^{\xint} \to \xF_{*}$ and the second holds because
  $\mathcal{O}^{\xint}$ and $\mathcal{P} \times_{\xF_{*}}
  \xF_{*}^{\xint}$ are the cocartesian
  fibrations over $\xF_{*}^{\xint}$ for the right Kan extensions of
  $\mathcal{O}^{\simeq}_{\angled{1}}$ and $\mathcal{P}_{\angled{1}}$
  along the inclusion $\{\angled{1}\} \hookrightarrow \xF_{*}$, respectively.

  It follows that for a symmetric monoidal \icat{}
  $\mathcal{C}^{\otimes}$ we have a natural equivalence
  \[ \Map_{\SMCI}(\Env(\mathcal{O}^{\xint}), \mathcal{C}^{\otimes})
    \simeq \Map(\mathcal{O}^{\simeq}_{\angled{1}}, \mathcal{C}).\]
  Thus $\Env(\mathcal{O}^{\xint})$ has the universal property of the
  free symmetric monoidal \icat{} on
  $\mathcal{O}^{\simeq}_{\angled{1}}$ (which is also the free
  symmetric monoidal \igpd{}).

  On the other hand, from the construction of $\Env$ we see that the
  symmetric monoidal functor $\Env(\mathcal{O}^{\xint}) \to
  \Env(\mathcal{O})$ induced by the inclusion of $\mathcal{O}^{\xint}$
  is an equivalence on underlying symmetric monoidal \igpds{}. This
  shows that the inclusion of $\mathcal{O}^{\simeq}_{\angled{1}}$
  exhibits the underlying symmetric monoidal \igpd{} of
  $\Env(\mathcal{O})$ as free, which is what we wanted to prove.
\end{proof}

\begin{cor}
  The functor $\Env \colon \OpdI \to \SMCI \simeq \CMon(\CatI)$ fits
  in a commutative square
  \[
    \begin{tikzcd}
      \OpdI \arrow{r}{\Env} \arrow{d}[swap]{(\blank)_{\angled{1}}^{\simeq}} & \CMon(\CatI) \arrow{d}{(\blank)^{\simeq}} \\
      \mathcal{S} \arrow{r}{\Sym} & \CMon(\mathcal{S}),
    \end{tikzcd}
  \]
  and so the functor $\Env$ factors uniquely through a functor
  $\Env'' \colon \OpdI \to \PROPI$ over $\mathcal{S}$. \qed
\end{cor}

\begin{lemma}
  The functor $\Env'' \colon \OpdI \to \PROPI$ has a right adjoint
  $U''$, which takes $(\mathcal{C}^{\otimes}, \Sym(X) \simeq
  \mathcal{C}^{\simeq})$ to the full suboperad of
  $\mathcal{C}^{\otimes}$ on the objects in the subspace $X \subseteq
  \Sym(X) \simeq \mathcal{C}^{\simeq}$. 
\end{lemma}
\begin{proof}
  Given $\mathcal{O} \in \OpdI$ and $(\mathcal{C}^{\otimes}, \alpha
  \colon \Sym(X)
  \simeq \mathcal{C}^{\simeq}) \in \PROPI$ we have a natural pullback
  square
  \[
    \begin{tikzcd}
      \Map_{\PROPI}(\Env''(\mathcal{O}),
      (\mathcal{C}^{\otimes},\alpha)) \arrow{r} \arrow{d} &
      \Map_{\SMCI}(\Env(\mathcal{O}), \mathcal{C}^{\otimes}) \arrow{d}
      \\
      \Map_{\mathcal{S}}(\mathcal{O}^{\simeq}_{\angled{1}}, X)
      \arrow{r} &
      \Map_{\CMon(\mathcal{S})}(\Env(\mathcal{O})_{\angled{1}}^{\simeq},
      \mathcal{C}^{\simeq}).
    \end{tikzcd}
  \]
  We can rewrite the right-hand part of this square
  using the adjunction $\Env \dashv U$ as well as
  the free--forgetful adjunction for commutative monoids as
  \[
    \begin{tikzcd}
      \Map_{\PROPI}(\Env''(\mathcal{O}),
      (\mathcal{C}^{\otimes},\alpha)) \arrow{r} \arrow{d} &
      \Map_{\OpdI}(\mathcal{O}, U\mathcal{C}^{\otimes}) \arrow{d}
      \\
      \Map_{\mathcal{S}}(\mathcal{O}^{\simeq}_{\angled{1}}, X)
      \arrow[hookrightarrow]{r} &
      \Map_{\mathcal{S}}(\mathcal{O}_{\angled{1}}^{\simeq},
      \mathcal{C}^{\simeq}),
    \end{tikzcd}
  \]
  where the bottom horizontal map is now the inclusion of those maps
  that factor through $X \hookrightarrow \Sym(X) \simeq \mathcal{C}^{\simeq}$.
  The pullback is then precisely the space of \iopd{} maps
  $\mathcal{O} \to U\mathcal{C}^{\otimes}$ that factor through the full
  suboperad $U''\mathcal{C}^{\otimes}$ on the objects in $X$, so that
  we have a natural equivalence
  \[\Map_{\PROPI}(\Env''(\mathcal{O}),
      (\mathcal{C}^{\otimes},\alpha)) \simeq \Map_{\OpdI}(\mathcal{O},
      U''(\mathcal{C}^{\otimes},\alpha)),\]
    as required.
\end{proof}

\begin{propn}
  $\Env'' \colon \OpdI \to \PROPI$ is fully faithful.
\end{propn}
\begin{proof}
  It suffices to show that the unit transformation $\id \to U''\Env''$
  is an equivalence, which follows from the pullback square
  \cref{eq:iOpb} in \cref{lem:iOcart}, since this exhibits
  $\mathcal{O}$ as the full suboperad of $\Env''(\mathcal{O})$ spanned
  by objects that lie over $\mathbf{1}$ in $\xF$, which are precisely
  the objects that lie in $\mathcal{O}^{\simeq}_{\angled{1}}$ under
  the equivalence $\Sym \mathcal{O}^{\simeq}_{\angled{1}} \simeq
  \Env''(\mathcal{O})_{\angled{1}}^{\simeq}$.
\end{proof}

\begin{cor}\label{cor:opd-counit}
  $\OpdI$ is equivalent to the full subcategory of $\PROPI$ consisting
  of pairs $(\mathcal{C}^{\otimes}, \alpha \colon \Sym(X) \simeq
  \mathcal{C}^{\simeq})$  such that the counit map
  \begin{equation}
    \label{eq:U''counit}
   \Env(U''(\mathcal{C}^{\otimes},\alpha)) \to
   \mathcal{C}^{\otimes}
  \end{equation}
  is an equivalence. \qed
\end{cor}

We can also give an explicit description of this subcategory:
\begin{propn}\label{propn:counit-hereditary}
  For $(\mathcal{C}^{\otimes}, \alpha \colon \Sym(X) \simeq
  \mathcal{C}^{\simeq})$ in $\PROPI$, the
  counit map \cref{eq:U''counit}  is an equivalence \IFF{} the
  following ``hereditary'' condition holds:
  \begin{enumerate}
  \item[\textup{($*$)}] Given objects $x_{1},\ldots,x_{n}$ and
      $y_{1},\ldots,y_{m}$ in $X \subseteq \mathcal{C}^{\simeq}$, the
      morphism
      \[ \coprod_{\phi \in
          \Map_{\xF}(\mathbf{n},\mathbf{m})} \prod_{i=1}^{m}
        \Map_{\mathcal{C}}\left(\bigotimes_{j \in \phi^{-1}(i)} x_{j},
          y_{i}\right) \to \Map_{\mathcal{C}}\left(\bigotimes_{j=1}^{n}
          x_{j}, \bigotimes_{i=1}^{m} y_{i}\right),\]
      given by tensoring maps together, is an equivalence.
  \end{enumerate}
\end{propn}
\begin{remark}
  The PROP structure together with the hereditary condition can be
  seen as an $\infty$-categorical version of what Kaufmann and
  Ward call \emph{Feynman 
  categories} 
  (\cite[Definition~1.1]{Kaufmann-Ward:1312.1269}; see also 
  \cite[3.2]{Batanin-Kock-Weber:1510.08934} for a version closer to ours).
  \cref{cor:opd-counit} and \cref{propn:counit-hereditary}
  together are then an $\infty$-categorical version of the equivalence established
  in \cite[5.16]{Batanin-Kock-Weber:1510.08934} between Feynman categories
  and operads.
\end{remark}

\begin{proof}[Proof of \cref{propn:counit-hereditary}]
  A morphism in $\PROPI$ is an equivalence \IFF{} it projects to an
  equivalence in both $\mathcal{S}$ and $\SMCI$. By construction the
  counit maps to an equivalence in $\mathcal{S}$, and the forgetful
  functor from $\SMCI$ to $\CatI$ is conservative, so it suffices to
  show the given condition is equivalent to the functor
  \[ \epsilon \colon \Env(U''\mathcal{C})_{\angled{1}} \to \mathcal{C} \]
  being an equivalence.

  From \cref{propn:Enveqisfree} we see that the underlying map of
  symmetric monoidal \igpds{} of $\epsilon$ is an equivalence, so that
  $\epsilon$ is in particular essentially surjective. This means we
  only need to show the given condition is equivalent to $\epsilon$
  being fully faithful. That in turn follows from identifying the
  mapping spaces in $\Env(U''\mathcal{C})_{\angled{1}} \simeq
  (U''\mathcal{C})^{\act}$ as in the proof of \cref{propn:iopdimage}. 
\end{proof}

\section{From Symmetric Monoidal $\infty$-Categories to Presheaves}
\label{sec:x}

Our goal in this section is to use the description of \iopds{} as a
full subcategory of $\SMC_{\infty/\xF^{\amalg}}$ from
\cref{propn:iopdimage} to give a presentation of \iopds{} as a
localization of a presheaf \icat{}. We first recall the description of
\icats{} as complete Segal spaces, and the more general notion of
Segal $\mathcal{O}$-spaces over an algebraic pattern $\mathcal{O}$, in
\S\ref{subsec:segsp}. Then we prove in \S\ref{subsec:sliceseg} that
overcategories in Segal $\mathcal{O}$-spaces can be described as
Segal spaces for another algebraic pattern. We apply this to symmetric
monoidal \icats{} in \S\ref{subsec:slsmc}, which in particular gives a
presentation of $\SMC_{\infty/\xF^{\amalg}}$, and then finally apply
this to describe \iopds{} in \S\ref{sec:appl-infty-oper}.

\subsection{Segal Spaces}\label{subsec:segsp}
Here we briefly recall Rezk's definition of \icats{} as complete Segal
spaces. As we will consider several similar structures, it is
convenient to do so using some terminology from \cite{patterns1}:
\begin{defn}
  An \emph{algebraic pattern} is an \icat{} $\mathcal{O}$ equipped
  with a factorization system (whereby every morphism factors as an
  \emph{inert} morphism followed by an \emph{active} morphism) and a
  collection of \emph{elementary objects}. We write
  $\mathcal{O}^{\xint}$ and $\mathcal{O}^{\act}$ for the subcategories
  containing only the inert and active maps, respectively, and
  $\mathcal{O}^{\el} \subseteq \mathcal{O}^{\xint}$ for the full
  subcategory of elementary objects and inert maps between them.
  
  The purpose of algebraic patterns is to be an abstract general
  setting for Segal conditions, as we proceed to explain. A
  basic example is the category $\Dop$ (as explained in
  \ref{Delta-algebraicpattern} below).
\end{defn}

\begin{notation}
  If $\mathcal{O}$ is an algebraic pattern, then for $X \in
  \mathcal{O}$ we write \[\mathcal{O}^{\el}_{X/} := \mathcal{O}^{\el}
  \times_{\mathcal{O}^{\xint}} \mathcal{O}^{\xint}_{X/}\] for the
  \icat{} of inert maps from $X$ to elementary objects.
\end{notation}

\begin{defn}
  Let $\mathcal{O}$ be an algebraic pattern and $\mathcal{C}$ an
  \icat{} with limits of shape $\mathcal{O}^{\el}_{X/}$ for all $X \in
  \mathcal{O}$. Then a \emph{Segal $\mathcal{O}$-object} in
  $\mathcal{C}$ is a functor $F \colon \mathcal{O} \to \mathcal{C}$
  such that for all $X \in \mathcal{O}$ the natural map
  \[ F(X) \to \lim_{E \in \mathcal{O}^{\el}_{X/}} F(E) \] is an
  equivalence. We call a Segal $\mathcal{O}$-object in the \icat{}
  $\mathcal{S}$ a \emph{Segal $\mathcal{O}$-space}.
\end{defn}
\begin{remark}
  Equivalently, a Segal $\mathcal{O}$-object is a functor $F \colon
  \mathcal{O} \to \mathcal{C}$ such that the restriction
  $F|_{\mathcal{O}^{\xint}}$ is a right Kan extension of $F|_{\mathcal{O}^{\el}}$.
\end{remark}

\begin{notation}
  If $\mathcal{O}$ is an algebraic pattern, we write
  $\Seg_{\mathcal{O}}(\mathcal{C})$ for the full subcategory of
  $\Fun(\mathcal{O}, \mathcal{C})$ spanned by the Segal
  $\mathcal{O}$-objects.
\end{notation}

\begin{ex}
  We consider the category $\xF_{*}$ as an algebraic pattern using the
  factorization system of \cref{defn:F*fact} and with $\angled{1}$ as
  the unique elementary object. Then a Segal $\xF_{*}$-object in an
  \icat{} $\mathcal{C}$ is precisely a commutative monoid in the sense
  of \cref{defn:cmon}.
\end{ex}

\begin{notation}
  We write $\simp$ for the simplex category, \ie{} the category of
  ordered sets $[n] := \{0 < 1 < \cdots < n\}$ ($n = 0,1,\ldots$) and
  order-preserving maps between them. A morphism $\phi \colon [n] \to
  [m]$ in $\simp$ is called
  \emph{inert} if it is a subinterval inclusion, \ie{} $\phi(i) =
  \phi(0)+i$ for all $i$, and \emph{active} if it preserves the end
  points, \ie{} $\phi(0) = 0$ and $\phi(n) = m$. The active and inert
  morphisms form a factorization system on $\simp$.
  For $0 \leq i \leq j \leq n$, we write $\iota_{ij} \colon [j-i] \hookrightarrow [n]$ for the inert
  map in $\simp$ given by $\iota_{ij}(t) = i+t$, \ie{} the inclusion
  of $\{i,i+1,\ldots,j\}$.
\end{notation}

\begin{ex}\label{Delta-algebraicpattern}
  We view $\Dop$ as an algebraic pattern using this inert--active
  factorization system, with $[0]$ and $[1]$ as the elementary
  objects. A Segal $\Dop$-object in an \icat{} $\mathcal{C}$ is then a
  simplicial object $F \colon \Dop \to \mathcal{C}$ such that the
  natural map
  \[ F([n]) \to F([1]) \times_{F([0])} \cdots \times_{F([0])} F([1])\]
  determined by the inert maps $[0],[1] \hookrightarrow [n]$ in
  $\simp$ (\ie{} the maps $\iota_{ii}$ and $\iota_{i(i+1)}$), is an
  equivalence. In particular, a Segal $\Dop$-space is precisely a
  \emph{Segal space} in the sense of Rezk~\cite{RezkCSS}.
\end{ex}

\begin{notation}
  Let $E^{1} \in \Seg_{\Dop}(\mathcal{S})$ denote the nerve of the
  generic equivalence, \ie{} the category with two objects and a
  unique morphism between any pair of objects. (Equivalently, this is
  the simplicial set with $n$-simplices $E^{1}_{n} =
  \{0,1\}^{\mathbf{n}}$.)
\end{notation}

\begin{defn}
  For $X \in \Seg_{\Dop}(\mathcal{S})$, an \emph{equivalence} in $X$
  is a morphism $E^{1}\to X$. We write $X^{\txt{eq}}:=
  \Map_{\Seg_{\Dop}(\mathcal{S})}(E^{1}, X)$ for the space of
  equivalences in $X$. The Segal space $X$ is \emph{complete} if the
  map $X_{0} \to X^{\txt{eq}}$ given by composition with $E^{1} \to
  \Delta^{0}$ is an equivalence. (In other words, $X$ is complete if
  it is local with respect to this morphism.) We write
  $\CSeg_{\Dop}(\mathcal{S}) \subseteq \Seg_{\Dop}(\mathcal{S})$ for
  the full subcategory spanned by the complete Segal spaces.
\end{defn}

\begin{thm}[Joyal--Tierney \cite{JoyalTierney}]\label{thm:CSSicat}
  The restricted Yoneda embedding
  \[\CatI \to \Fun(\Dop, \mathcal{S})\] along the functor
  $\simp \to \CatI$ given by viewing the partially ordered sets $[n]$
  as ($\infty$-)categories, induces an equivalence
  \[ \CatI \isoto \CSeg_{\Dop}(\mathcal{S}).\]
\end{thm}

\subsection{Slices via Segal Conditions}\label{subsec:sliceseg}
In this subsection we prove that if $B$ is a Segal $\mathcal{O}$-space
then we can describe the overcategory
$\Seg_{\mathcal{O}}(\mathcal{S})_{/B}$ as the \icat{} of Segal
$\mathcal{B}$-spaces, where $\mathcal{B} \to \mathcal{O}$ is the left
fibration corresponding to $B$.

The starting point is the following observation:
\begin{propn}[\cite{freepres}*{Corollary 9.8}]\label{propn:lkanlfibeq}
  Let $\mathcal{B}$ be an \icat{} and let $\pi \colon \mathcal{E} \to
  \mathcal{B}$ be a left fibration. Then the
  functor
  \[\pi_{!} \colon \Fun(\mathcal{E}, \mathcal{S}) \to \Fun(\mathcal{B},
    \mathcal{S}),\]
  given by left Kan extension along $\pi$, induces an equivalence
  \begin{equation}
    \label{eq:lkaneqsl}
   \Fun(\mathcal{E}, \mathcal{S}) \isoto
    \Fun(\mathcal{B},\mathcal{S})_{/E}, 
  \end{equation}
  where the value of $\pi_{!}$ at the terminal object is the functor
  $E \colon \mathcal{B} \to \mathcal{S}$ corresponding to the left
  fibration $\pi$. \qed
\end{propn}
\begin{remark}
  Under the straightening equivalence between
  $\Fun(\mathcal{B}, \mathcal{S})$ and
  $\Cat_{\infty/\mathcal{B}}^{\txt{L}}$, the functor $\pi_{!}$ is
  given by composition with $\pi$, and the equivalence
  \cref{eq:lkaneqsl} boils down to the observation that if we have a
  commutative triangle
  \[
    \begin{tikzcd}
      \mathcal{X} \arrow{rr}{f} \arrow{dr} & & \mathcal{E} \arrow{dl}{\pi} \\
      & \mathcal{B}
    \end{tikzcd}
  \]
  then $f$ is a left fibration \IFF{} $\pi f$ is a left fibration.
\end{remark}

\begin{defn}\label{def:lfibpatt}
  Let $\mathcal{O}$ be an algebraic pattern, and suppose
  $\pi \colon \mathcal{B} \to \mathcal{O}$ is a left fibration. Then
  $\mathcal{B}$ inherits a factorization system where the inert and
  active morphisms are simply those that lie over inert and active
  morphisms in $\mathcal{O}$. If the functor
  $B \colon \mathcal{O} \to \mathcal{S}$ corresponding to $\pi$ is a
  Segal $\mathcal{O}$-space, we view $\mathcal{B}$ as an algebraic
  pattern via this factorization system and with all objects that lie
  over elementary objects in $\mathcal{O}$ as its elementary objects.
\end{defn}

\begin{remark}\label{rmk:lfibsegcond}
  Suppose $\mathcal{O}$ is an algebraic pattern and $\mathcal{B} \to
  \mathcal{O}$ is the left fibration corresponding to a Segal
  $\mathcal{O}$-space. Then for every object $\overline{X}$ in
  $\mathcal{B}$ lying over $X$ in $\mathcal{O}$, the functor $\pi$
  induces an equivalence
  \[ \mathcal{B}_{\overline{X}/}^{\el} \isoto \mathcal{O}^{\el}_{X/}, \]
  since there is a unique (cocartesian) morphism over every (inert)
  morphism $X \to E$ in $\mathcal{O}$. This means that
  a functor $F \colon
  \mathcal{B} \to \mathcal{S}$ is a Segal
  $\mathcal{B}$-space \IFF{} for every $\overline{X} \in \mathcal{B}$ lying
  over $X \in \mathcal{O}$ the natural map
  \[ F(\overline{X}) \to \lim_{E \in \mathcal{O}^{\el}_{X/}}
    F(\overline{E}) \]
  is an equivalence, where $\overline{X} \to \overline{E}$ is the
  cocartesian morphism lying over $X \to E$.
\end{remark}

\begin{propn}\label{propn:Segsliceeq}
  Let $\mathcal{O}$ be an algebraic pattern, and suppose $\pi \colon
  \mathcal{B} \to \mathcal{O}$ is a left fibration corresponding to a
  Segal $\mathcal{O}$-space $B$. Then the equivalence
  \cref{eq:lkaneqsl} from \cref{propn:lkanlfibeq} restricts to 
  an equivalence
  \[ \pi_{!} \colon \Seg_{\mathcal{B}}(\mathcal{S}) \isoto
    \Seg_{\mathcal{O}}(\mathcal{S})_{/B}.\]
\end{propn}
\begin{proof}
  Since we know from \cref{propn:lkanlfibeq} that $\pi_{!}$ gives an
  equivalence \[\Fun(\mathcal{B},\mathcal{S}) \isoto \Fun(\mathcal{O},
  \mathcal{S})_{/B},\] it suffices to show that the full subcategories
  of Segal objects are identified under this equivalence.
  For $F \colon \mathcal{B} \to \mathcal{S}$ and $X \in \mathcal{O}$,
  we have a commutative square
  \[
    \begin{tikzcd}
      \pi_{!}F(X) \arrow{r} \arrow{d} & \lim_{E \in
        \mathcal{O}^{\el}_{X/}} \pi_{!}F(E) \arrow{d} \\
      B(X) \arrow{r}{\sim} & \lim_{E \in \mathcal{O}^{\el}_{X/} } B(E).
    \end{tikzcd}
  \]
  The functor $\pi_{!}F$ is a Segal $\mathcal{O}$-space \IFF{}
  the top horizontal morphism is an equivalence in every such
  square. Since $B$ is a Segal $\mathcal{O}$-space, we know that the
  bottom horizontal morphism is an equivalence, and hence this
  condition is equivalent to all these squares being 
  pullbacks. This
  in turn is equivalent to the map on fibres over every point of
  $B(X)$ being an equivalence for every $X \in \mathcal{O}$.

  Since limits commute, we can
  identify the map on fibres over $p \in B(X)$ as
  \begin{equation}
    \label{eq:fibreseg}
   \pi_{!}F(X)_{p} \to \lim_{E \in
     \mathcal{O}^{\el}_{X/}}\pi_{!}F(E)_{p_{E}},   
 \end{equation}
 where $p_{E}$ is the image of $p$ in $B(E)$ under the map
 corresponding to $X \to E$ in $\mathcal{O}$.
 
  Since the functor $\pi$ is a left fibration, the left Kan extension
  $\pi_{!}$ is computed fibrewise, \ie{}
  \[ \pi_{!}F(X) \simeq \colim_{p \in B(X)} F(p).\]
  For a space $T$, the straightening equivalence $\Fun(T, \mathcal{S})
  \isoto \mathcal{S}_{/T}$ is given by taking colimits, with inverse
  given by taking fibres. Hence we have a natural identification of
  $\pi_{!}F(X)_{p}$ with $F(p)$, under which the map
  \cref{eq:fibreseg} corresponds to the Segal map
  \[ F(p) \to \lim_{E \in \mathcal{O}^{\el}_{X/}} F(p_{E}).\]
  As we saw in \cref{rmk:lfibsegcond}, asking for this to be an
  equivalence for all $X \in \mathcal{O}$ and $p \in B(X)$ is
  precisely asking for $F$ to be a Segal $\mathcal{B}$-space.
\end{proof}

In the special case where $\mathcal{O}$ is $\Dop$, we can use
\cref{propn:Segsliceeq} to get a description of the overcategory
$\Cat_{\infty/\mathcal{C}}$ in terms of complete Segal conditions;
this description can also be found in \cite{AyalaFrancisFib} and
\cite{HinichYoneda}.

\begin{remark}
  Let $\pi \colon \mathcal{B} \to \Dop$ be a left fibration
  corresponding to a Segal space $B$. An object $b \in \mathcal{B}$
  over $[n] \in \Dop$ corresponds to a morphism between left
  fibrations $i_{b} \colon \Dop_{/[n]} \to \mathcal{B}$ over $\Dop$.
  For $b \in \mathcal{B}_{0}$, composition with $i_{b}$ then
  restricts to a functor
  \[i_{b}^{*} \colon \Seg_{\mathcal{B}}(\mathcal{S}) \to
    \Seg_{\Dop}(\mathcal{S}),\]
  since through the equivalence of
  \cref{propn:Segsliceeq} the functor $i_{b}^{*}$ corresponds to
  base change along $i_{b}$, which preserves the Segal condition since
  limits commute. 
\end{remark}

\begin{defn}\label{defn:lfibcompl}
  Let $\pi \colon \mathcal{B} \to \Dop$ be a left fibration
  corresponding to a Segal space $B$.
  We say a Segal $\mathcal{B}$-space $F$ is
  \emph{complete} if the Segal spaces $i_{b}^{*}F$ are complete for
  all $b \in \mathcal{B}_{0}$, or equivalently the fibres $(\pi_{!}F)_{b}$
  are all complete. We write $\CSeg_{\mathcal{B}}(\mathcal{S})
  \subseteq \Seg_{\mathcal{B}}(\mathcal{S})$ for
  the full subcategory of complete Segal $\mathcal{B}$-spaces.
\end{defn}

\begin{propn}\label{propn:lfibcompl}
  Let $\pi \colon \mathcal{B} \to \Dop$ be a left fibration
  corresponding to a simplicial space $B$. If $B$ is a complete Segal
  space, then the functor $\pi_{!}$ restricts to an equivalence
  \[ \CSeg_{\mathcal{B}}(\mathcal{S}) \isoto \CSeg_{\Dop}(\mathcal{S})_{/B}.\]
\end{propn}
\begin{proof}
  Suppose $X$ is a Segal space over $B$. We then have a commutative
  square
  \[
    \begin{tikzcd}
      X_{0} \arrow{r} \arrow{d} & X^{\txt{eq}} \arrow{d} \\
      B_{0} \arrow{r}{\sim} & B^{\txt{eq}},
    \end{tikzcd}
  \]
  where the bottom horizontal morphism is an equivalence since $B$ is
  complete. The Segal space $X$ is therefore complete \IFF{} this
  square of spaces is a pullback, which is equivalent to the map on
  fibres over each $b \in B_{0}$ being an equivalence. Thus $X$ is
  complete \IFF{} for each $b \in B_{0}$ the map on fibres $X_{b,0} \to
  (X^{\txt{eq}})_{b}$ is an equivalence. Since $\Map(E^{1},\blank)$
  preserves limits, we can also identify the fibre
  $(X^{\txt{eq}})_{b}$ with $(X_{b})^{\txt{eq}}$, so this condition
  says precisely that the Segal spaces
  $X_{b}$ are complete for all $b \in B_{0}$.
\end{proof}

Combining this observation with \cref{thm:CSSicat}, we get:
\begin{cor}
  For $\mathcal{C}$ an \icat{}, let
  \[ \simp_{/\mathcal{C}} := \simp \times_{\CatI}
    \Cat_{\infty/\mathcal{C}} \to \simp\] be the right fibration
  corresponding to $\mathcal{C}$ viewed as a complete Segal
  space. Then the restricted Yoneda embedding along
  $\simp_{/\mathcal{C}} \to \Cat_{\infty/\mathcal{C}}$ induces an equivalence
\[ \Cat_{\infty/\mathcal{C}} \isoto
  \CSeg_{\Dop_{/\mathcal{C}}}(\mathcal{S}).\] 
\end{cor}

\subsection{Slices of Symmetric Monoidal $\infty$-Categories}\label{subsec:slsmc}
We now specialize the results of the previous section to describe the
overcategories of $\SMCI$ in terms of Segal and completeness
conditions. We first observe that $\SMCI$ itself admits such a
description:
\begin{defn}
  We view the product $\xF_{*} \times \Dop$ as an algebraic pattern
  with the inert and active morphisms given by those that are inert
  and active in each coordinate, and with $(\angled{1},[0])$ and
  $(\angled{1},[1])$ as the elementary objects. We say that a Segal $(\xF_{*}
  \times \Dop)$-space $F$ is \emph{complete} if the Segal space
  $F(\angled{1}, \blank)$ is complete. We write $\CSeg_{\xF_{*} \times
    \Dop}(\mathcal{S}) \subseteq \Seg_{\xF_{*} \times
    \Dop}(\mathcal{S})$ for the full subcategory of complete Segal
  $(\xF_{*} \times \Dop)$-spaces.
\end{defn}

\begin{propn}\label{SMCisFDSeg}
  The restricted Yoneda embedding along $\simp \to \CatI$ induces an
  equivalence
  \[ \SMCI \isoto \CSeg_{\xF_{*} \times \Dop}(\mathcal{S}).\]
\end{propn}
\begin{proof}
  It follows immediately from the definitions that the equivalence
  \[ \Fun(\xF_{*}, \Fun(\Dop, \mathcal{S})) \simeq \Fun(\xF_{*} \times
    \Dop, \mathcal{S}) \]
  restricts to an equivalence
  \[ \Seg_{\xF_{*}}(\Seg_{\Dop}(\mathcal{S})) \simeq \Seg_{\xF_{*}
      \times \Dop}(\mathcal{S}).\]
  Moreover, since complete Segal spaces are closed under limits, this
  restricts further to an equivalence
  \[\Seg_{\xF_{*}}(\CSeg_{\Dop}(\mathcal{S})) \simeq \CSeg_{\xF_{*}
      \times \Dop}(\mathcal{S}).\]
  In other words, complete Segal $(\xF_{*} \times \Dop)$-spaces are commutative
  monoids in $\CSeg_{\Dop}(\mathcal{S})$. Combining this with the
  equivalence of \cref{thm:CSSicat} now gives the result.
\end{proof}

\begin{remark}\label{rmk:smcseg}
  Let $\mathcal{M} \to \xF_{*} \times \Dop$ be a left fibration
  corresponding to a Segal $(\xF_{*} \times \Dop)$-space $M$. Then
  a functor $F \colon \mathcal{M} \to \mathcal{S}$ is a Segal
  $\mathcal{M}$-space \IFF{} for $X \in \mathcal{M}$ lying over $(\angled{k},[n])$
  in $\xF_{*} \times \Dop$ the natural map
  \[ F(X) \to \prod_{i=1}^{k} F(X_{i,01}) \times_{F(X_{i,1})} \cdots \times_{F(X_{i,n-1})}
    F(X_{i,(n-1)n}) \]
  induced by the (cocartesian) maps $X \to X_{i,j}$ over
  $\rho_{i}\times \iota_{jj}$
  and $X \to X_{i,(j-1)j}$ over $\rho_{i} \times \iota_{(j-1)j}$, is
  an equivalence. This condition can conveniently be split into three parts:
  \begin{enumerate}[(1)]
  \item $F(X) \isoto F(X_{01}) \times_{F(X_{1})} \cdots
    \times_{F(X_{n-1})} F(X_{(n-1)n})$ where $X$ lies over $[n] \in
    \Dop$ and the maps $X \to X_{j}$ and $X \to X_{(j-1)j}$ are
    cocartesian over $\iota_{jj}$ and $\iota_{(j-1)j}$, respectively.
  \item $F(X) \isoto \prod_{i=1}^{k} F(X_{i})$ where $X$ lies over
    $(\angled{k}, [1])$ and $X \to X_{i}$ is cocartesian over
    $\rho_{i}$,
   \item $F(X) \isoto \prod_{i=1}^{k} F(X_{i})$ where $X$ lies over
    $(\angled{k}, [0])$ and $X \to X_{i}$ is cocartesian over
    $\rho_{i}$.
  \end{enumerate}
\end{remark}

As a special case of \cref{propn:Segsliceeq} we have:
\begin{cor}\label{cor:SMsliceeq}
  Let $\pi \colon \mathcal{M} \to \xF_{*} \times \Dop$ be a left fibration
  corresponding to a Segal $(\xF_{*} \times \Dop)$-space $M$.
  Then the functor
  $\pi_{!}$ given by left Kan extension along $\pi$ restricts to an equivalence
  \[ \Seg_{\mathcal{M}}(\mathcal{S}) \isoto \Seg_{\xF_{*}\times \Dop}(\mathcal{S})_{/M}.\]
\end{cor}

We now want to incorporate completeness into this description:

\begin{defn}
    Let $\pi \colon \mathcal{M} \to \xF_{*} \times \Dop$ be a left fibration
  corresponding to a Segal $(\xF_{*} \times \Dop)$-space $M$, and let
  $\mathcal{M}_{\angled{1}} \to \Dop$ be the fibre at
  $\angled{1} \in \xF_{*}$, corresponding to the underlying Segal
  space $M(\angled{1},\blank)$. Let $u_{M} \colon
  \mathcal{M}_{\angled{1}} \to \mathcal{M}$ denote the inclusion of
  this fibre; composition with $u_{M}$ restricts to a functor
  $\Seg_{\mathcal{M}}(\mathcal{S}) \to
  \Seg_{\mathcal{M}_{\angled{1}}}(\mathcal{S})$. We say a Segal
  $\mathcal{M}$-space $F$ is \emph{complete} if $u_{M}^{*}F$ is
  complete in the sense of \cref{defn:lfibcompl}, and write
  $\CSeg_{\mathcal{M}}(\mathcal{S}) \subseteq
  \Seg_{\mathcal{M}}(\mathcal{S})$ for the full subcategory spanned by
  the complete objects.
\end{defn}

\begin{propn}\label{propn:csegsymmon}
  Let $\pi \colon \mathcal{M} \to \xF_{*} \times \Dop$ be a left fibration
  corresponding to a Segal $(\xF_{*} \times \Dop)$-space $M$.
  If $M$ is complete, then the functor $\pi_{!}$ restricts to an equivalence
  \[ \CSeg_{\mathcal{M}}(\mathcal{S}) \isoto \CSeg_{\xF_{*} \times
      \Dop}(\mathcal{S})_{/M}.\]
\end{propn}
\begin{proof}
  Since left Kan extensions along the left fibration $\pi$ are given
  by taking colimits fibrewise, we have a commutative square
  \[
    \begin{tikzcd}
      \Seg_{\mathcal{M}}(\mathcal{S}) \arrow{d}{\pi_{!}}
      \arrow{r}{u_{M}^{*}} &
      \Seg_{\mathcal{M}_{\angled{1}}}(\mathcal{S})
      \arrow{d}{\pi_{\angled{1},!}} \\
      \Seg_{\xF_{*} \times \Dop}(\mathcal{S})_{/M} \arrow{r} &
      \Seg_{\Dop}(\mathcal{S})_{/M(\angled{1},\blank)},
    \end{tikzcd}
  \]
  where the vertical maps are equivalences.
  Combining this observation with \cref{propn:lfibcompl} now completes
  the proof, since $\CSeg_{\mathcal{M}}(\mathcal{S})$ and $\CSeg_{\xF_{*} \times
      \Dop}(\mathcal{S})_{/M}$ are defined as the preimages in this
    diagram of
    $\CSeg_{\mathcal{M}_{\angled{1}}}(\mathcal{S})$ and
    $\CSeg_{\Dop}(\mathcal{S})_{/M(\angled{1},\blank)}$, respectively.
\end{proof}

\begin{cor}\label{cor:SMCslCSeg}
  Suppose $\mathcal{C}^{\otimes}$ is a symmetric monoidal \icat{}, and
  let $\mathcal{M} \to \xF_{*}\times \Dop$ be the left fibration
  corresponding to $\mathcal{C}^{\otimes}$ viewed as a commutative
  monoid in complete Segal spaces. Then there is an equivalence
  of \icats{}
  \[ \SMC_{\infty/\mathcal{C}^{\otimes}} \simeq \CSeg_{\mathcal{M}}(\mathcal{S}).\]
\end{cor}
\begin{proof}
  Combine \cref{propn:csegsymmon} with \cref{SMCisFDSeg}.
\end{proof}

\subsection{Application to $\infty$-Operads}
\label{sec:appl-infty-oper}

Our next goal is to combine the results of the previous subsection
with those of \S\ref{subsec:envimg} to obtain a new description of
$\OpdI$ in terms of Segal and completeness conditions.

\begin{defn}
  Let $\pi \colon \mathcal{F} \to \xF_{*} \times \Dop$ be the left
  fibration corresponding to the symmetric monoidal category
  $\xF^{\amalg}$ viewed as a commutative monoid in Segal
  spaces. Unwinding the definitions, the category
  $\mathcal{F}$ has the following explicit description (where it is
  convenient to use the description of $\xF_{*}$ in terms of spans of finite
  sets): The objects of $\mathcal{F}$ are sequences of maps in $\xF$
  \[ 
    \begin{tikzcd}
      \mathbf{a}_{0} \arrow{r} \arrow{dr} & \cdots
      \arrow{r} & \mathbf{a}_{m} \arrow{dl} \\
      & \mathbf{n},
    \end{tikzcd}
  \]
  where this object lives over $(\mathbf{n}, [m])$ in $\xF_{*} \times
  \Dop$. A (necessarily cocartesian) morphism over $(\mathbf{n} \hookleftarrow \mathbf{x} \to
  \mathbf{n'}, [m'] \xto{\phi} [m])$ with this object as source is given by
  a commutative diagram
  \[
    \begin{tikzcd}[column sep=small,row sep=small]
      \mathbf{b}_{0} \arrow{dddr} \arrow{dr} \arrow[hookrightarrow]{rrr} & & &
      \mathbf{a}_{\phi(0)} \arrow{dddr}
      \arrow{dr} \\
       & \cdots \arrow{dr} & & & \cdots \arrow{dr} \\
       & & \mathbf{b}_{m'} \arrow{dl} \arrow[crossing over,hookrightarrow]{rrr} & & &
       \mathbf{a}_{\phi(m')} \arrow{dl} \\
       & \mathbf{x} \arrow{d} \arrow[hookrightarrow]{rrr} & & & \mathbf{n} \\
       & \mathbf{n}'
    \end{tikzcd}
  \]
  where the squares
  \[
    \begin{tikzcd}
      \mathbf{b}_{i}  \arrow{d} \arrow[hookrightarrow]{r} & \mathbf{a}_{\phi(i)}
      \arrow{d} \\
      \mathbf{x} \arrow[hookrightarrow]{r} & \mathbf{n}
    \end{tikzcd}
  \]
  are all pullback squares. (In other words, we restrict along $\phi$, pull
  back along $\mathbf{x} \hookrightarrow \mathbf{n}$, and compose with
  the map $\mathbf{x} \to \mathbf{n'}$.)
\end{defn}

\begin{remark}\label{rmk:FSeg}
  With this description of $\mathcal{F}$, the requirements for a
  functor $\Phi \colon \mathcal{F} \to \mathcal{S}$ to be a Segal
  $\mathcal{F}$-space from \cref{rmk:smcseg} amount to the
  following maps being equivalences:
  \begin{equation}
    \label{eq:FSeg1}\!
    \Phi\left(
      \begin{tikzcd}[cramped,column sep=tiny]
      \mathbf{a}_{0} \arrow{r} \arrow{dr} & \cdots
      \arrow{r} & \mathbf{a}_{m} \arrow{dl} \\
      & \mathbf{n}
    \end{tikzcd}
  \right) \to
  \Phi\left(
    \begin{tikzcd}[cramped,column sep=0em]
      \mathbf{a}_{0} \arrow{rr} \arrow{dr} & & \mathbf{a}_{1}
      \arrow{dl} \\
       & \mathbf{n}
     \end{tikzcd}\right)
   \!\times_{\!\!\!\!\Phi\left(
       \begin{tikzcd}[cramped,column sep=tiny,font=\tiny]
         \mathbf{a}_{1} \arrow{d} \\
         \mathbf{n}
       \end{tikzcd}
     \right)}
   \!\cdots
   \times_{\!\!\!\!\Phi\left(
       \begin{tikzcd}[cramped,column sep=tiny,font=\tiny]
         \mathbf{a}_{m-1} \arrow{d} \\
         \mathbf{n}
       \end{tikzcd}
     \right)}\!
  \Phi\left(
    \begin{tikzcd}[cramped,column sep=0em]
      \mathbf{a}_{m-1} \arrow{rr} \arrow{dr} & & \mathbf{a}_{m}
      \arrow{dl} \\
       & \mathbf{n}
     \end{tikzcd}\right)\!,
 \end{equation}
 \begin{equation}
   \label{eq:FSeg2}
     \Phi\left(
    \begin{tikzcd}[cramped,column sep=tiny]
      \mathbf{a} \arrow{rr} \arrow{dr} & & \mathbf{b}
      \arrow{dl} \\
       & \mathbf{n}
     \end{tikzcd}\right)
   \to
   \prod_{i = 1}^{n}      \Phi\left(
    \begin{tikzcd}[cramped,column sep=tiny]
      \mathbf{a}_{i} \arrow{rr} \arrow{dr} & & \mathbf{b}_{i}
      \arrow{dl} \\
       & \mathbf{1}
     \end{tikzcd}\right)\!,
 \end{equation}
 \begin{equation}
   \label{eq:FSeg3}
   \Phi\left(
     \begin{tikzcd}[cramped,column sep=tiny]
       \mathbf{a} \arrow{d} \\
       \mathbf{n}
     \end{tikzcd}
   \right)
   \to
   \prod_{i=1}^{n}    \Phi\left(
     \begin{tikzcd}[cramped,column sep=tiny]
       \mathbf{a}_{i} \arrow{d} \\
       \mathbf{1}
     \end{tikzcd}
   \right)\!.
 \end{equation}
\end{remark}

From \cref{cor:SMCslCSeg} we then get the following:
\begin{cor}\label{cor:smcxFcseg}
  There is an equivalence
  \[ \SMC_{\infty/\xF^{\amalg}} \isoto
    \CSeg_{\mathcal{F}}(\mathcal{S}),\]
  where the right-hand side is the full subcategory of
  $\Fun(\mathcal{F}, \mathcal{S})$ spanned by functors $\Phi$ satisfying
  conditions \cref{eq:FSeg1}--\cref{eq:FSeg3} and for which the Segal space
  $\Phi_{\angled{1},\mathbf{a}}$ is complete for every $\mathbf{a} \in
  \xF$. \qed
\end{cor}

\begin{defn}
  We write  $\Fun'(\mathcal{F}, \mathcal{S})$ for the full subcategory
  of $\Fun(\mathcal{F}, \mathcal{S})$ spanned by
  functors $\Phi$ such that for every object
    \[ 
    \begin{tikzcd}
      \mathbf{a}_{0} \arrow{r} \arrow{dr} & \cdots
      \arrow{r} & \mathbf{a}_{m} \arrow{dl} \\
      & \mathbf{n},
    \end{tikzcd}
  \]
    in $\mathcal{F}$ and every map $\mathbf{n} \to \mathbf{n}'$ in $\xF$, the map
  \begin{equation}
    \label{eq:primecond}
    \Phi\left(
      \begin{tikzcd}
      \mathbf{a}_{0} \arrow{r} \arrow{dr} & \cdots
      \arrow{r} & \mathbf{a}_{m} \arrow{dl} \\
      & \mathbf{n}
    \end{tikzcd}
  \right) \to
  \Phi\left(
      \begin{tikzcd}
      \mathbf{a}_{0} \arrow{r} \arrow{dr} & \cdots
      \arrow{r} & \mathbf{a}_{m} \arrow{dl} \\
      & \mathbf{n}'
    \end{tikzcd}
  \right) 
\end{equation}
lying over $\mathbf{n} \xfrom{=} \mathbf{n} \to \mathbf{n}'$
is an equivalence.
We then write $\Seg'_{\mathcal{F}}(\mathcal{S})$
and $\CSeg'_{\mathcal{F}}(\mathcal{S})$ for the intersections of
$\Fun'(\mathcal{F},\mathcal{S})$ with the full subcategories
$\Seg_{\mathcal{F}}(\mathcal{S})$ and
$\CSeg_{\mathcal{F}}(\mathcal{S})$ in $\Fun(\mathcal{F},\mathcal{S})$,
respectively.
\end{defn}

Our goal is now to prove the following:
\begin{thm}\label{thm:opdascsegF}
  The equivalence of \cref{cor:smcxFcseg} restricts along the fully
  faithful inclusion $\Env' \colon \OpdI \hookrightarrow
  \SMC_{\infty/\xF^{\amalg}}$ from \cref{EnvoverFff} to an equivalence
  \[ \OpdI \isoto \CSeg'_{\mathcal{F}}(\mathcal{S}).\]
\end{thm}

We begin by simplifying the definition of
$\Seg'_{\mathcal{F}}(\mathcal{S})$ a bit:
\begin{lemma}\label{lem:segprimesimple}
  Suppose $\Phi$ is in $\Seg_{\mathcal{F}}(\mathcal{S})$. Then $\Phi$
  lies in $\Seg'_{\mathcal{F}}(\mathcal{S})$ \IFF{} the two following
  conditions hold:
  \begin{enumerate}[(1)]
\item For every object $\mathbf{n}$ in $\xF$, the map
  \[  \Phi\left(
      \begin{tikzcd}[cramped,column sep=tiny]
        \mathbf{n} \arrow[equals]{d} \\
        \mathbf{n}
      \end{tikzcd}
    \right) \to \Phi\left(
      \begin{tikzcd}[cramped,column sep=tiny]
        \mathbf{n} \arrow{d} \\
        \mathbf{1}
      \end{tikzcd}
    \right)
  \]
  over $ \mathbf{n} \xfrom{=}
  \mathbf{n} \to \mathbf{1}$
  is an equivalence.
    \item For every morphism $\mathbf{n} \to \mathbf{m}$ in $\xF$, the
    map
    \[ \Phi\left(
        \begin{tikzcd}[cramped,column sep=tiny]
          \mathbf{n} \arrow{rr} \arrow{dr} & & \mathbf{m}
          \arrow[equals]{dl} \\
          & \mathbf{m}
        \end{tikzcd}\right) \to \Phi\left(
        \begin{tikzcd}[cramped,column sep=tiny]
          \mathbf{n} \arrow{rr} \arrow{dr} & & \mathbf{m}
          \arrow{dl} \\
          & \mathbf{1}
        \end{tikzcd}\right)
    \]
    over $\mathbf{m} \xfrom{=}
	\mathbf{m} \to \mathbf{1}$
    is an equivalence.
\end{enumerate}
\end{lemma}
\begin{proof}
  Clearly (1) and (2) are special cases of \cref{eq:primecond}, so we
  need to prove that these special cases suffice. We first observe that
  condition \cref{eq:FSeg1} for $\Seg_{\mathcal{F}}(\mathcal{S})$
  implies that $\Phi$ lies in $\Seg'_{\mathcal{F}}(\mathcal{S})$
  \IFF{} condition \cref{eq:primecond} holds for $m = 0$ and $m =
  1$. Now we claim that (1) and (2) are equivalent to these two cases,
  respectively; we will prove the case where $m = 1$, the proof for
  $m = 0$ being similar.

  For any morphism $\mathbf{n} \to \mathbf{n}'$ in $\xF$, consider the
  maps
  \[       \Phi\left(
        \begin{tikzcd}[cramped,column sep=tiny]
          \mathbf{a} \arrow{rr} \arrow{dr} & & \mathbf{b}
          \arrow{dl} \\
          & \mathbf{n}
        \end{tikzcd}\right) \to 
      \Phi\left(
        \begin{tikzcd}[cramped,column sep=tiny]
          \mathbf{a} \arrow{rr} \arrow{dr} & & \mathbf{b}
          \arrow{dl} \\
          & \mathbf{n}'
        \end{tikzcd}\right)  \to 
      \Phi\left(
        \begin{tikzcd}[cramped,column sep=tiny]
          \mathbf{a} \arrow{rr} \arrow{dr} & & \mathbf{b}
          \arrow{dl} \\
          & \mathbf{1}
        \end{tikzcd}\right).
    \]
    where the first map lies over
    $\mathbf{n} \xfrom{=}
	\mathbf{n} \to \mathbf{n}'$ and the second
    lies over $\mathbf{n}' \xfrom{=}
	\mathbf{n}' \to
    \mathbf{1}$. Then the composite lies over
    $\mathbf{n} \xfrom{=}
	\mathbf{n} \to \mathbf{1}$, so that by the
    2-of-3 property of equivalences, condition
    \cref{eq:primecond} holds for all maps \IFF{} it holds for maps of
    the form $\mathbf{n} \to \mathbf{1}$.

    Next, consider for any map $\mathbf{b} \to \mathbf{n}$ in $\xF$ the maps
    \[ \Phi\left(
        \begin{tikzcd}[cramped,column sep=tiny]
          \mathbf{a} \arrow{rr} \arrow{dr} & & \mathbf{b}
          \arrow[equals]{dl} \\
          & \mathbf{b}
        \end{tikzcd}\right) \to 
      \Phi\left(
        \begin{tikzcd}[cramped,column sep=tiny]
          \mathbf{a} \arrow{rr} \arrow{dr} & & \mathbf{b}
          \arrow{dl} \\
          & \mathbf{n}
        \end{tikzcd}\right)  \to 
      \Phi\left(
        \begin{tikzcd}[cramped,column sep=tiny]
          \mathbf{a} \arrow{rr} \arrow{dr} & & \mathbf{b}
          \arrow{dl} \\
          & \mathbf{1}
        \end{tikzcd}\right).
    \]
    Here assumption (2) implies that the composite map is an
    equivalence, so that the second map is an equivalence \IFF{} the
    first one is. But the first map decomposes using the Segal
    conditions as
    \[  \prod_{i=1}^{n}     \Phi\left(
        \begin{tikzcd}[cramped,column sep=tiny]
          \mathbf{a}_{i} \arrow{rr} \arrow{dr} & & \mathbf{b}_{i}
          \arrow[equals]{dl} \\
          & \mathbf{b}_{i}
        \end{tikzcd}\right) \to  \prod_{i=1}^{n}
      \Phi\left(
        \begin{tikzcd}[cramped,column sep=tiny]
          \mathbf{a}_{i} \arrow{rr} \arrow{dr} & & \mathbf{b}_{i}
          \arrow{dl} \\
          & \mathbf{1}
        \end{tikzcd}\right),\]
    which is also an equivalence under assumption (2). Thus the Segal
    conditions and (2) imply that \cref{eq:primecond} holds in the
    case $m = 1$, as required.
\end{proof}

\begin{proof}[Proof of \cref{thm:opdascsegF}]
  We must show that under the equivalence of \cref{cor:smcxFcseg}, the
  two conditions from \cref{propn:iopdimage} correspond precisely to
  \cref{eq:primecond}. Equivalently, we can check that the alternative
  conditions ($1''$) and ($2''$) from \cref{rmk:iopdcondprime} correspond to those from
  \cref{lem:segprimesimple}. To this end, let $\mathcal{C}^{\otimes}
  \to \xF^{\amalg}$ be an object of $\SMC_{\infty/\xF^{\amalg}}$ and
  $\Phi$ the corresponding object in
  $\CSeg_{\mathcal{F}}(\mathcal{S})$. Unwinding the definitions, we
  have equivalences
  \[       \Phi\left(
        \begin{tikzcd}[cramped,column sep=tiny]
          \mathbf{n} \arrow{d} \\
          \mathbf{1}
        \end{tikzcd}\right) \simeq \mathcal{C}_{(n)}^{\simeq}, \qquad
  \Phi\left(
        \begin{tikzcd}[cramped,column sep=tiny]
          \mathbf{n} \arrow{rr}{\phi} \arrow{dr} & & \mathbf{m}
          \arrow{dl} \\
          & \mathbf{1}
        \end{tikzcd}\right)
      \simeq \Map(\Delta^{1}, \mathcal{C})_{\phi}, \]
    under which the tensoring maps
 \[ (\mathcal{C}^{\simeq}_{(1)})^{\times n} \to
   \mathcal{C}^{\simeq}_{(n)}, \qquad \prod_{i=1}^{m} \Map(\Delta^{1},
   \mathcal{C})_{\mathbf{n}_{i} \to \mathbf{1}} \to
   \Map(\Delta^{1}, \mathcal{C})_{\phi} \]
 from conditions ($1''$) and
 ($2''$) correspond to
 \[ \prod_{i=1}^{n} \Phi\left(
        \begin{tikzcd}[cramped,column sep=tiny]
          \mathbf{1} \arrow[equals]{d} \\
          \mathbf{1}
        \end{tikzcd}\right) \isofrom \Phi\left(
        \begin{tikzcd}[cramped,column sep=tiny]
          \mathbf{n} \arrow[equals]{d} \\
          \mathbf{n}
        \end{tikzcd}\right) \to \Phi\left(
        \begin{tikzcd}[cramped,column sep=tiny]
          \mathbf{n} \arrow{d} \\
          \mathbf{1}
        \end{tikzcd}\right)\]
    and
\[ \prod_{i=1}\Phi\left(
        \begin{tikzcd}[cramped,column sep=tiny]
          \mathbf{n_{i}} \arrow{rr}{\phi} \arrow{dr} & & \mathbf{1}
          \arrow[equals]{dl} \\
          & \mathbf{1}
        \end{tikzcd}\right) \isofrom
      \Phi\left(\begin{tikzcd}[cramped,column sep=tiny]
          \mathbf{n} \arrow{rr}{\phi} \arrow{dr} & & \mathbf{m}
          \arrow[equals]{dl} \\
          & \mathbf{m}
        \end{tikzcd}\right)
      \to \Phi\left(
        \begin{tikzcd}[cramped,column sep=tiny]
          \mathbf{n} \arrow{rr}{\phi} \arrow{dr} & & \mathbf{m}
          \arrow{dl} \\
          & \mathbf{1}
        \end{tikzcd}\right),
      \]
  respectively. Here we have exactly the maps from
  \cref{lem:segprimesimple}, so that the conditions there
  precisely correspond to those from \cref{rmk:iopdcondprime}, as required.
\end{proof}

\cref{thm:opdascsegF} has the following immediate corollary:
\begin{cor}\label{cor:pres}
  The $\infty$-category $\OpdI$ is presentable.
\end{cor}

The following observation will be useful later:
\begin{lemma}\label{lem:Seg'comp}
  An object $\Phi \in \Seg'_{\mathcal{F}}(\mathcal{S})$ lies in
  $\CSeg'_{\mathcal{F}}(\mathcal{S})$ \IFF{} the Segal space
  $\Phi_{\angled{1},\mathbf{1}}$ is complete.
\end{lemma}
\begin{proof}
  For $\mathbf{n} \in \xF$ we have a natural zig-zag of simplicial
  spaces
  \[ \prod_{i=1}^{n} \Phi\left(
        \begin{tikzcd}[cramped,column sep=tiny]
          \mathbf{1} \arrow[equals]{r} \arrow[equals]{dr} & \cdots
          \arrow[equals]{r} & \mathbf{1}
          \arrow[equals]{dl} \\
          & \mathbf{1}
        \end{tikzcd}\right)
      \from
      \Phi\left(
        \begin{tikzcd}[cramped,column sep=tiny]
          \mathbf{n} \arrow[equals]{r} \arrow[equals]{dr} & \cdots
          \arrow[equals]{r} & \mathbf{n}
          \arrow[equals]{dl} \\
          & \mathbf{n}
        \end{tikzcd}\right) \to
      \Phi\left(
        \begin{tikzcd}[cramped,column sep=tiny]
          \mathbf{n} \arrow[equals]{r} \arrow{dr} & \cdots
          \arrow[equals]{r} & \mathbf{n}
          \arrow{dl} \\
          & \mathbf{1}
        \end{tikzcd}\right),
    \]
    where both maps are equivalences for $\Phi$ in
    $\Seg'_{\mathcal{F}}(\mathcal{S})$. Thus we have an equivalence
    between the Segal spaces $\Phi_{\angled{1},\mathbf{1}}^{\times n}$
    and $\Phi_{\angled{1},\mathbf{n}}$. Since complete Segal spaces
    are closed under limits, this implies that if
    $\Phi_{\angled{1},\mathbf{1}}$ is complete then so 
    $\Phi_{\angled{1},\mathbf{n}}$, and hence by definition $\Phi$
    lies in $\CSeg'_{\mathcal{F}}(\mathcal{S})$.
\end{proof}

\section{From Barwick's $\infty$-Operads to Symmetric Monoidal
  $\infty$-Categories}
In this section we first review Barwick's model of \iopds{} in
\S\ref{subsec:baropd}. Then in \S\ref{subsec:DFtoF} we use our work in
the previous section to give a new proof of the equivalence between
Barwick's and Lurie's models by passing through the equivalence of
\cref{thm:opdascsegF}.

\subsection{Barwick's $\infty$-Operads}\label{subsec:baropd}
Here we recall Barwick's definition of \iopds{} from
\cite{BarwickOpCat} (there called \emph{complete Segal operads}). This
definition can be phrased as complete Segal spaces for a certain
algebraic pattern, which we introduce first:

\begin{defn}
  The category $\DF$ has as objects pairs $([n] \in \simp, f \colon [n]
  \to \xF)$ interpreted as chains (of length $n$) of composable arrows in 
  $\xF$,
  and morphisms $([n], f) \to ([m], g)$ are given by
  morphisms $\phi \colon [n] \to [m]$ in $\simp$ together with a
  natural transformation $\eta \colon f \to g \circ \phi$ such that
  \begin{itemize}
  \item for every $i \in [n]$, the map $\eta_{i} \colon f(i) \to g(i)$
    is an injection,
  \item for every $i, j$ in $[n]$ with $i \leq j$, the commutative
    square
    \[
      \begin{tikzcd}
        f(i) \arrow{r} \arrow[hookrightarrow]{d}{\eta_{i}} & f(j)
        \arrow[hookrightarrow]{d}{\eta_{j}} \\
        g(i) \arrow{r} & g(j)
      \end{tikzcd}
    \]
    is a pullback.
  \end{itemize}
  For small values of $n$, we shall also write out an object
  $([n],f)$ as a chain
  \[ f(0) {\to}\cdots{\to} f(n).\]
\end{defn}
Note that the projection $\DF \to \simp$ is a cartesian fibration. We
can lift the active-inert factorization system on $\simp$ to one on
$\DF$ by declaring a map $(\phi, \eta) \colon ([n],f) \to ([m],g)$ to
be
\begin{itemize}
\item \emph{active} if $\phi$ is active in $\simp$ and $\eta_{i}
  \colon f(i) \to g(i)$ is an isomorphism for all $i$,
\item \emph{inert} if $\phi$ is inert in $\simp$.
\end{itemize}
This gives a factorization system on $\DF$ compatible with that on $\simp$.
\begin{remark}
  Given $(\phi, \eta) \colon ([n],f) \to ([m],g)$ where $\phi \colon [n]
  \to [m]$ factors as $[n] \xto{a} [k] \xto{i} [m]$ with $a$ active and
  $i$ inert, to find the active-inert factorization in $\DF$ we first take a factorization of $(\phi, \eta)$ as
  \[ ([n], f) \to ([k], gi) \to ([m], g) \] with the second morphism
  cartesian, and then a factorization of $([n],f) \to ([k],gi)$ as
  $([n], f) \to ([k], g') \to ([k], gi)$ where $g'$ is given by taking
  pullbacks along $f(n) \to gi(a(n)) = gi(k)$.
\end{remark}
\begin{defn}
  We give $\DFop$ the structure of an algebraic pattern using the
  inert-active factorization system we just defined, and with the
  elementary objects being the $1$-chains $\mathbf{n}{\to}\mathbf{1}$ for all
  $\mathbf{n}$ in $\xF$ as well as the $0$-chain $\mathbf{1}$.
\end{defn}

\begin{remark}\label{rmk:DFSeg}
  A functor $F \colon \DFop \to \mathcal{S}$ is a Segal $\DFop$-space
  \IFF{} the following three conditions hold:
  \begin{enumerate}[(1)]
  \item $F([n], f) \isoto F([1], f_{01}) \times_{F(f(1))} \cdots
    \times_{F(f(n-1))} F([1], f_{(n-1)n})$,
  \item $F(\mathbf{a}{\to}\mathbf{b}) \isoto \prod_{i \in
      \mathbf{b}} F(\mathbf{a}_{i}{\to}\mathbf{1})$,
  \item $F(\mathbf{b}) \isoto \prod_{i \in
      \mathbf{b}} F(\mathbf{1})$.    
  \end{enumerate}
\end{remark}

\begin{remark}
Segal $\DFop$-objects describe the algebraic structure of \iopds{}:
\begin{itemize}
\item $F(\mathbf{1})$ is the space of objects,
\item $F(\mathbf{n}{\to}\mathbf{1})$ is the space of $n$-ary
  operations, with the map
  \[ F(\mathbf{n}{\to}\mathbf{1}) \to F(\mathbf{1})^{\times n
 } \times F(\mathbf{1})\] coming from the $n+1$ inclusions
 $(\mathbf{1}) \to (\mathbf{n}{\to}\mathbf{1})$ assigning to
 each operation its sources and target,
\item $F(\mathbf{n}{\to}\mathbf{m}{\to}\mathbf{1})$ decomposes
  under the Segal condition as the space of $n_{i}$-ary operations
  that can be composed with an $m$-ary operation,
\item the map $F(\mathbf{n}{\to}\mathbf{m}{\to}\mathbf{1}) \to
  F(\mathbf{n}{\to}\mathbf{1})$ induced by the inner face map
  $d_{1}$ encodes composition,
\item and the remaining data encodes the homotopy-coherent
  associativity and unitality of this composition operation.
\end{itemize}
\end{remark}

To complete the definition we also need to add a completeness
condition:
\begin{defn}
  Let $u \colon \Dop \to \DFop$ be the functor given by
  \[ [n] \,\,\mapsto\,\, ([n], \mathbf{1}\xto{=}\mathbf{1}\xto{=}
  \cdots
  \xto{=}\mathbf{1}).\] Composition
  with $u$ gives a functor $u^{*} \colon \Seg_{\DFop}(\mathcal{S}) \to
  \Seg_{\Dop}(\mathcal{S})$, and we say $F \in
  \Seg_{\DFop}(\mathcal{S})$ is \emph{complete} if $u^{*}F$ is a
  complete Segal space. We write $\CSeg_{\DFop}(\mathcal{S})$ for the
  full subcategory of $\Seg_{\DFop}(\mathcal{S})$ spanned by the
  complete Segal objects. 
\end{defn}

\subsection{Comparison}\label{subsec:DFtoF}
Our goal is now to show that $\CSeg_{\DFop}(\mathcal{S})$ is
equivalent to the \icat{} $\CSeg'_{\mathcal{F}}(\mathcal{S})$
considered in the previous section, where
$\pi \colon \mathcal{F} \to \xF_{*} \times \Dop$ is the left fibration
corresponding to the symmetric monoidal category $\xF^{\amalg}$ viewed as a commutative monoid in Segal spaces. As a first step, we
see that there is a functor relating $\mathcal{F}$ to $\DFop$:

\begin{defn}
  We define $P \colon \mathcal{F} \to \DFop$ on the object
  \[ 
    \begin{tikzcd}
      \mathbf{a}_{0} \arrow{r} \arrow{dr} & \cdots
      \arrow{r} & \mathbf{a}_{m} \arrow{dl} \\
      & \mathbf{n},
    \end{tikzcd}
  \]
  by forgetting the ``augmentation'' to $\mathbf{n}$, so that $P$
  takes this object to $\mathbf{a}_{0}{\to}\cdots{\to}\mathbf{a}_{m}$.
  Comparing the definitions of the morphisms in
  $\mathcal{F}$ and $\DFop$, we see that a morphism in $\mathcal{F}$
  restricts to a morphism in $\DFop$ when we forget the augmentations,
  which gives the action of $P$ on morphisms.
\end{defn}

\begin{remark}\label{rmk:Pcoc}
  The functor $P$ fits in a commutative triangle
  \[
    \begin{tikzcd}
      \mathcal{F}  \arrow{dr} \arrow{rr}{P} & & \DFop \arrow{dl} \\
       & \Dop,
    \end{tikzcd}
  \]
  where both maps to $\Dop$ are cocartesian fibrations. From the
  definitions of the cocartesian morphisms we also see that $P$
  preserves these.
\end{remark}

The key observation is the following:
\begin{propn}\label{propn:Ploc}
  The functor $P \colon \mathcal{F} \to \DFop$ is a localization, and
  composition with it gives an equivalence
  \[ \Fun(\DFop, \mathcal{S}) \isoto \Fun'(\mathcal{F}, \mathcal{S}). \]
\end{propn}

We begin by looking at $P$ on each fibre over $\Dop$:
\begin{defn}
  For $[m] \in \simp$, let $S_{m} \colon \simp_{\xF,[m]}^{\op} \to
  \mathcal{F}_{[m]}$ be the functor given by 
 taking the object $\mathbf{a}_{0}{\to}\cdots{\to}\mathbf{a}_{m}$ to
  \[
    \begin{tikzcd}
      \mathbf{a}_{0} \arrow{r} \arrow{dr} & \cdots
      \arrow{r} & \mathbf{a}_{m} \arrow[equals]{dl} \\
      & \mathbf{a}_{m},
    \end{tikzcd}
  \]
  and a morphism $(\mathbf{a}_{0}{\to}\cdots{\to}\mathbf{a}_{m}) \to
  (\mathbf{b}_{0}{\to}\cdots{\to}\mathbf{b}_{m})$ in $\DF$ given
  by $\eta \colon \mathbf{a}_{(\blank)} \to
  \mathbf{b}_{(\blank)}$ to the morphism in $\mathcal{F}$ given by
  pulling back along $\mathbf{a}_{m}
  \hookrightarrow \mathbf{b}_{m}$. 
\end{defn}

\begin{lemma}\label{lem:Pmloc}
  Let $P_{m}$ be the restriction of $P$ to the
  fibre over $[m] \in \Dop$.
  \begin{enumerate}[(i)]
  \item The functor  $S_{m}$ is left adjoint to $P_{m}$.
  \item $P_{m}$ is a localization.
  \end{enumerate}
\end{lemma}
\begin{proof}
  We have $P_{m} S_{m} = \id$ by inspection. We can define a natural
  transformation $\alpha\colon S_{m}P_{m} \to \id_{\mathcal{F}}$ given at the
  object
    \[
    \begin{tikzcd}
      \mathbf{a}_{0} \arrow{r} \arrow{dr} & \cdots
      \arrow{r} & \mathbf{a}_{m} \arrow{dl} \\
      & \mathbf{n},
    \end{tikzcd}
  \]
  by the map from
  \[
    \begin{tikzcd}
      \mathbf{a}_{0} \arrow{r} \arrow{dr} & \cdots
      \arrow{r} & \mathbf{a}_{m} \arrow[equals]{dl} \\
      & \mathbf{a}_{m}
    \end{tikzcd}
  \]
  lying over $ \mathbf{a}_{m} \xfrom{=}
  \mathbf{a}_{m} \to \mathbf{n}$
  (given by composing with $\mathbf{a}_{m}\to \mathbf{n}$). Then
  $\alpha S_{m}$ and $P_{m}\alpha$ are clearly both the respective identity
  transformations, so this indeed exhibits $S_{m}$ as the left
  adjoint of $P_{m}$. This proves (i). Moreover, since $P_{m}\alpha$
  is the identity we see that $\alpha$ becomes a natural isomorphism
  after we invert the morphisms in $\mathcal{F}_{[m]}$ that are taken
  to isomorphisms by $P_{m}$. This means that after localizing, $S_{m}$
  is an inverse of $P_{m}$, which proves (ii).
\end{proof}

To prove \cref{propn:Ploc} we use the following criterion:
\begin{propn}[Hinich]\label{propn:Hinich}
  Suppose we have a commutative triangle
  \[
    \begin{tikzcd}
      \mathcal{E}  \arrow{dr}{p} \arrow{rr}{f} && \mathcal{E}' \arrow{dl}{p'}
      \\
       & \mathcal{B}
    \end{tikzcd}
  \]
  where $p$ and $p'$ are cocartesian fibrations and $f$ preserves
  cocartesian morphisms. If for every $b \in \mathcal{B}$ the functor
  on fibres $f_{b} \colon \mathcal{E}_{b} \to \mathcal{E}'_{b}$ is a
  localization, then so is $f$.
\end{propn}
\begin{proof}
  Let $W_{b}$ denote the collection of morphisms in $\mathcal{E}_{b}$
  that are taken to equivalences by $p_{b}$; since $f$ preserves
  cocartesian morphisms we have for every map $\beta \colon b \to b'$
  a commutative square
  \[
    \begin{tikzcd}
      \mathcal{E}_{b} \arrow{d} \arrow{r}{\beta_{!}} &
      \mathcal{E}_{b'} \arrow{d} \\
      \mathcal{E}'_{b} \arrow{r}{\beta_{!}} & \mathcal{E}'_{b'},
    \end{tikzcd}
  \]
  from which it is immediate that $\beta_{!}W_{b} \subseteq W_{b'}$.   
  Unstraightening, we see that $f$ corresponds to the natural
  localization maps $\mathcal{E}_{b} \to \mathcal{E}_{b}[W_{b}^{-1}]$.
  It follows from Hinich's work on localizations of fibrations in
  \cite{HinichLoc} that $\mathcal{E}'$ is then the localization of
  $\mathcal{E}$ at the union of the $W_{b}$'s, which is to say at the
  maps that $f$ takes to equivalences. (More precisely, we apply
  \cite{HinichLoc}*{Proposition 2.1.4} in the form
  \cite{paradj}*{Proposition 4.2.5}.)
\end{proof}

\begin{proof}[Proof of \cref{propn:Ploc}]
  We saw in \cref{rmk:Pcoc} that $P$
  preserves cocartesian morphisms over $\Dop$ and in \cref{lem:Pmloc}
  that fibrewise $P_{m}$ is a localization for every $[m] \in
  \Dop$. \cref{propn:Hinich} then implies that $P$ is also a
  localization. If $W$ denotes the collection of morphisms in
  $\mathcal{F}$ that are taken to isomorphisms in $\DFop$, then it follows that composition with $P$ gives a fully
  faithful functor
  \[ P^{*} \colon \Fun(\DFop, \mathcal{S}) \to \Fun(\mathcal{F},
    \mathcal{S}) \]
  whose image is spanned by the functors $\mathcal{F} \to \mathcal{S}$
  that take the morphisms in $W$ to equivalences in $\mathcal{S}$.
  We can identify the morphisms in $W$ as those morphisms in
  $\mathcal{F}$ that lie over an identity in $\Dop$ and over a map of
  the form $ \mathbf{n}  \xfrom{=}
  \mathbf{n} \to \mathbf{n}''$ in
  $\xF_{*}$.   By definition, $\Fun'(\mathcal{F}, \mathcal{S})$ is
  the full subcategory of functors that take these morphisms to equivalences,
  and so it is precisely the image of $P^{*}$, as required.
\end{proof}

\begin{cor}\label{cor:FDFcomp}
  Composition with $P$ induces equivalences
    \[ P^{*} \colon
  \Seg_{\DFop}(\mathcal{S}) \isoto \Seg'_{\mathcal{F}}(\mathcal{S}),\]
  \[ P^{*} \colon
  \CSeg_{\DFop}(\mathcal{S}) \isoto \CSeg'_{\mathcal{F}}(\mathcal{S}).\]
\end{cor}
\begin{proof}
  We want to show that these subcategories correspond to each other under the
  equivalence of \cref{propn:Ploc}. In other words, we must show that
  a functor $\Phi \colon \DFop \to \mathcal{S}$ lies in
  $\Seg_{\DFop}(\mathcal{S})$ \IFF{} $P^{*}\Phi$ lies in
  $\Seg'_{\mathcal{F}}(\mathcal{S})$, and similarly for completeness.
  For the Segal conditions this is clear since the conditions in 
  \cref{rmk:FSeg} applied to $P^{*}\Phi$ give precisely the Segal
  conditions in \cref{rmk:DFSeg}, while for completeness this follows
  similarly using the simplified condition from \cref{lem:Seg'comp}.
\end{proof}

Combining \cref{cor:FDFcomp} with \cref{thm:opdascsegF} we have a
zig-zag of equivalences
\[ \OpdI \isoto \CSeg'_{\mathcal{F}}(\mathcal{S}) \isofrom
  \CSeg_{\DFop}(\mathcal{S}),\]
which gives:
\begin{cor}
  There is an equivalence of \icats{}
  \[ \OpdI \simeq \CSeg_{\DFop}(\mathcal{S})\]
  between Lurie's and Barwick's models for \iopds{}. \qed
\end{cor}

\end{document}